\documentclass[a4paper, reqno, 10pt]{amsart}

\usepackage{verbatim}
\usepackage{kotex}
\usepackage[dvips]{color}
\usepackage[usenames,dvipsnames]{xcolor}
\usepackage{amssymb}
\usepackage[active]{srcltx}
\usepackage[colorlinks,pdfpagelabels,pdfstartview = FitH,bookmarksopen = true,bookmarksnumbered = true,linkcolor = NavyBlue,plainpages = false,hypertexnames = false,citecolor = OliveGreen] {hyperref}
\usepackage[italian, textsize=tiny, color=BlueGreen, backgroundcolor=BlueGreen, linecolor=BlueGreen]{todonotes}
\usepackage{mathtools}
\usepackage{empheq}
\usepackage{a4wide}
\usepackage{aligned-overset}

\newcommand\blfootnote[1]{%
\begingroup
\renewcommand\thefootnote{}\footnote{#1}%
\addtocounter{footnote}{-1}%
\endgroup
}

\allowdisplaybreaks

\title[Mixed-order nonlinear fractional equations]{Regularity for mixed-order nonlinear fractional equations with degenerate coefficients} 

\author[Lee]{Ho-Sik Lee}
\address{Fakult\"at für Mathematik, Universität Bielefeld, 33615 Bielefeld, Germany}
\email{ho-sik.lee@uni-bielefeld.de}

\author[Ok]{Jihoon Ok}
\address{Department of Mathematics, Sogang University, Seoul 04107, Republic of Korea}
\email{jihoonok@sogang.ac.kr}

\author[Song]{Kyeong Song}
\address{School of Mathematics, Korea Institute for Advanced Study, Seoul 02455, Republic of Korea}
\email{kyeongsong@kias.re.kr}

\makeatletter
\@namedef{subjclassname@2020}{\textup{2020} Mathematics Subject Classification}
\makeatother

\subjclass[2020]{
35R11; 
47G20; 
35D30;  
35B65; 
35R05. 
}

\keywords{Nonlinear nonlocal operator; mixed order; degenerate coefficient; regularity; Harnack inequality}

\newtheorem{theorem}{Theorem}[section]

\newtheorem{lemma}[theorem]{Lemma}

\theoremstyle{definition}

\newtheorem{remark}[theorem]{Remark}

\numberwithin{equation}{section}

\def\eqn#1$$#2$${\begin{equation}\label#1#2\end{equation}}
\def\charfn_#1{{\raise1.2pt\hbox{$\chi_{\kern-1pt\lower3pt\hbox{{$\scriptstyle#1$}}}$}}}

\makeatletter
\newcommand{\pushright}[1]{\ifmeasuring@#1\else\omit\hfill$\displaystyle#1$\fi\ignorespaces}
\newcommand{\pushleft}[1]{\ifmeasuring@#1\else\omit$\displaystyle#1$\hfill\fi\ignorespaces}
\makeatother

\newcommand{\R}{\mathbb{R}}

\DeclareMathOperator*{\data}{\mathtt{data}}
\def\diam{\operatorname{diam}}

\DeclareMathOperator*{\esssup}{ess\,sup}
\DeclareMathOperator*{\essinf}{ess\,inf}
\DeclareMathOperator*{\essosc}{ess\,osc}

\newcommand{\ern}{\mathbb{R}^n}

\def\loc{{\operatorname{loc}}}

\newcommand{\supp}{{\rm supp}}

\def\mean#1{\mathchoice%
          {\mathop{\kern 0.2em\vrule width 0.6em height 0.69678ex depth -0.58065ex
                  \kern -0.8em \intop}\nolimits_{\kern -0.4em#1}}%
          {\mathop{\kern 0.1em\vrule width 0.5em height 0.69678ex depth -0.60387ex
                  \kern -0.6em \intop}\nolimits_{#1}}%
          {\mathop{\kern 0.1em\vrule width 0.5em height 0.69678ex
              depth -0.60387ex
                  \kern -0.6em \intop}\nolimits_{#1}}%
          {\mathop{\kern 0.1em\vrule width 0.5em height 0.69678ex depth -0.60387ex
                  \kern -0.6em \intop}\nolimits_{#1}}}

\newtoks\by
\newtoks\paper
\newtoks\book
\newtoks\jour
\newtoks\yr
\newtoks\pages
\newtoks\vol
\newtoks\publ

\def\ota{{\hbox{\bf ???}}}
\def\cLear{\by=\ota\paper=\ota\book=\ota\jour=\ota\yr=\ota
\pages=\ota\vol=\ota\publ=\ota}
\def\endpaper{\the\by, \textit{\the\paper},
{\the\jour} \textbf{\the\vol} (\the\yr), \the\pages.\cLear}
\def\endbook{\the\by, \textit{\the\book},
\the\publ, \the\yr.\cLear}
\def\endpap{\the\by, \textit{\the\paper}, \the\jour.\cLear}
\def\endproc{\the\by, \textit{\the\paper}, \the\book, \the\publ,
\the\yr, \the\pages.\cLear}

\begin{document}
\begin{abstract} 
We consider a class of nonlinear integro-differential equations whose leading operator is obtained as a superposition of $(-\Delta_{p})^{s}$ and $(-\Delta_{p})^{t}$, where $0<s<t<1<p<\infty$, weighted via two possibly degenerate coefficients $a(\cdot,\cdot),b(\cdot,\cdot) \ge 0$. 
We prove local boundedness and H\"older regularity of its weak solutions under natural assumptions on the coefficients $a(\cdot,\cdot)$, $b(\cdot,\cdot)$ and the powers $s,t$, and $p$.
Moreover, when $a(\cdot,\cdot) \equiv 1$, we also prove a Harnack inequality for weak solutions.
\end{abstract}

\maketitle

\blfootnote{H.-S. Lee was supported by the Deutsche Forschungsgemeinschaft (DFG, German Research Foundation) through (GRK 2235/2
2021 - 282638148) at Bielefeld University. J. Ok was supported by the National Research Foundation of Korea(NRF) grant
funded by the Korea government(MSIT) (RS-2025-24533680).
  K. Song was supported by a KIAS individual grant (MG091701) at Korea Institute for Advanced Study. }

\section{Introduction}
This paper is concerned with a class of nonlinear nonlocal equations whose differentiability order shows a drastic change depending on the point. More precisely, with $\Omega\subset \mathbb{R}^{n}$ being a bounded domain, we consider the following equation
\begin{equation}\label{main.eq}
\mathcal{L}u = 0 \quad \text{in } \Omega,
\end{equation}
whose leading operator $\mathcal{L}(\cdot)$ is given by
\begin{equation*}
\begin{aligned}
\mathcal{L}u(x) &\coloneqq \mathrm{P.V.}\int_{\mathbb{R}^{n}} |u(x)-u(y)|^{p-2}(u(x)-u(y))\left[a(x,y)K_{sp}(x,y)+b(x,y)K_{tp}(x,y)\right]\,dy
\end{aligned}
\end{equation*}
for $x\in\ern$. Here, $K_{sp},K_{tp}: \mathbb{R}^{n}\times\mathbb{R}^{n}\rightarrow \mathbb{R}$ are measurable kernels of orders $(s,p)$ and $(t,p)$, respectively, for some $0<s < t <1< p <\infty$, and $a,b:\mathbb{R}^{n}\times\mathbb{R}^{n}\rightarrow \mathbb{R}$ are nonnegative functions. We refer to Section~\ref{setting} below for the detailed assumptions. 

There have been several researches on various kinds of anisotropic nonlocal equations, see for instance  
 \cite{CK20,FKV15,KW24} for the linear case and \cite{BOS,CK22} for the nonlinear case.   
 The equation \eqref{main.eq} is modeled on the following example, which is the case where $K_{sp}(x,y) \equiv |x-y|^{-n-sp}$ and $K_{tp}(x,y) \equiv |x-y|^{-n-tp}$:
\begin{equation}\label{model}
\begin{aligned}
&\mathrm{P.V.}\int_{\mathbb{R}^{n}} |u(x)-u(y)|^{p-2}(u(x)-u(y))\left(\frac{a(x,y)}{|x-y|^{n+sp}}+\frac{b(x,y)}{|x-y|^{n+tp}}\right)\,dy = 0 \quad \text{in }\Omega.
\end{aligned}
\end{equation}
It is straightforward to see that \eqref{model} is the Euler-Lagrange equation of the functional
\begin{equation}\label{functional}
v \mapsto \iint_{\mathcal{C}_{\Omega}}\left(a(x,y)\frac{|v(x)-v(y)|^{p}}{|x-y|^{n+sp}} + b(x,y)\frac{|v(x)-v(y)|^{p}}{|x-y|^{n+tp}}\right)\,dx\,dy,
\end{equation}
where
\begin{equation}\label{COmega}
\mathcal{C}_{\Omega} \coloneqq (\mathbb{R}^{n}\times\mathbb{R}^{n})\setminus((\mathbb{R}^{n}\setminus\Omega)\times(\mathbb{R}^{n}\setminus\Omega)).
\end{equation}
A main point in \eqref{main.eq} is that the (possibly degenerate) coefficients $a(\cdot,\cdot)$ and $b(\cdot,\cdot)$ make the associated integro-differential operator $\mathcal{L}(\cdot)$ switch between two different fractional elliptic phases. 
In this respect, our problem is closely related to the following local equation 
\[ \mathrm{div}\left(a(x)|Du|^{p-2}Du + b(x)|Du|^{q-2}Du \right) = 0, \]
where $1 < p < q < \infty$ and $a(\cdot), b(\cdot) \ge 0$ satisfy $\nu \le a(x) + b(x) \le L$ for some constants $0 < \nu \le L < \infty$. This is called a double phase problem with two modulating coefficients. Since the pioneering works of Colombo and Mingione \cite{CM15a, CM15b}, the regularity theory for local double phase problems with a single modulating coefficient (i.e., the case $a(\cdot)\equiv 1$ in the equation above) has been extensively developed; see, for instance, \cite{BCM1, BCM2, CM16, DM20} and the references therein. More general classes of problems, including the double phase type, have subsequently been investigated in \cite{HO22, KO24}. In particular, \cite{KO24} studies various regularity results for equations of the form above, including the case in which $a(\cdot)$ may be degenerate, meaning that $a(\cdot)$ can vanish. A key idea in \cite{KO24} is the following observation: when $a(x)$ is close to zero, the equation can be regarded as a $q$‑Laplace equation with a lower‑order perturbation of $p$‑Laplace type, whereas when $a(x)$ stays away from zero, the equation behaves genuinely as a double phase problem with a single modulating coefficient.

We now turn to regularity results for nonlinear nonlocal equations. The De Giorgi-Nash-Moser theory for the $(s,p)$-Laplacian was first investigated by Di Castro, Kuusi and Palatucci \cite{DKP14,DKP16}; they proved local regularity and Harnack inequality for weak solutions to fractional $p$-Laplacian type equations involving measurable kernels. Cozzi \cite{coz} extended such results to minimizers of non-differentiable functionals with lower-order dependencies, via a slightly different approach using fractional De Giorgi classes. We  also refer to \cite{BDNS,Lind} and references therein for similar results concerned with more general equations.

Very recently, the results and techniques in \cite{coz,DKP14,DKP16} were further developed and extended to nonlocal problems with nonstandard growth and/or differentiability conditions. Similarly to the case of local problems, one can consider several typical examples, such as Orlicz growth \cite{BKO,CKW2}, variable growth \cite{CK23,Ok23}, and double phase \cite{BOS,DP19,fang2021weak}. These papers are concerned with local boundedness and H\"older continuity of weak solutions. 
Moreover, in the case of Orlicz growth, Harnack inequalities were also proved in \cite{BKS,CKW1}. 
We also mention the paper \cite{OS} concerned with regularity results and Harnack inequalities for nonlinear nonlocal equations with kernels of general differentiability order.

Nonlocal double phase problems with one modulating coefficient, whose prototype is
\begin{equation}\label{model.one}
\begin{aligned}
\mathrm{P.V.}\int_{\mathbb{R}^{n}} \left(\frac{|u(x)-u(y)|^{p-2}(u(x)-u(y))}{|x-y|^{n+sp}}+ b(x,y)\frac{|u(x)-u(y)|^{q-2}(u(x)-u(y))}{|x-y|^{n+tq}}\right)\,dy = 0
\end{aligned}
\end{equation}
with $s,t \in (0,1)$ and $1<p\le q < \infty$, were first considered by De Filippis and Palatucci \cite{DP19}. More precisely, they proved H\"older continuity of viscosity solutions for nonhomogeneous equations with bounded source terms. We also refer to \cite{BKK,fang2021weak,SM} for various results for \eqref{model.one}. The papers \cite{BKK,DP19,fang2021weak,SM} consider solutions bounded in $\mathbb{R}^{n}$ and are under the assumption that $t\le s$, which implies that the second term in \eqref{model.one} plays a role as a lower order term. Under these settings, $b(\cdot,\cdot)$ is imposed to be bounded and discontinuous.

On the contrary, in the paper \cite{BOS}, joint with Byun, the second and third authors of this paper proved the local boundedness and H\"older continuity of weak solutions to \eqref{model.one} in the case $s\le t$, under natural assumptions on the powers and the modulating coefficients. Specifically, it was proved that if
\[  b \in C^{0,\alpha}(\mathbb{R}^{n}\times\mathbb{R}^{n}) \;\; \text{for some}\;\; \alpha\in (0,1] \quad \text{and} \quad tq \le sp + \alpha, \]
then every locally bounded weak solution to \eqref{model.one} is locally H\"older continuous. 
We emphasize that when $s \le t$, the second term in \eqref{model} or \eqref{model.one} has a higher order and, as in the local case, the interplay between the H\"older regularity of $b(\cdot,\cdot)$ and the growth/differentiability condition of the problem plays a key role in the analysis. 
Moreover, in light of the Lavrentiev type phenomena considered in \cite{Balci}, 
the above assumption is essentially sharp.
We also refer to the paper \cite{BLS} concerned with regularity results and Harnack inequality for mixed local and nonlocal double phase problems.

In this paper, we prove the local boundedness and H\"older continuity of weak solutions to the general equation \eqref{main.eq}. Moreover, in the case of one modulating coefficient (i.e. when $a(\cdot,\cdot)\equiv 1$), we further prove a Harnack inequality for weak solutions to \eqref{main.eq}. 
To our knowledge, each of our regularity results is the first one for \eqref{main.eq}; 
moreover, our Harnack inequality given in Theorem~\ref{thm.harnack} below is the first one for nonlocal nonautonomous problems.

\subsection{Assumptions and main results}\label{setting}

We say that a function $F:\mathbb{R}^{n}\times\mathbb{R}^{n}\rightarrow \mathbb{R}$ is symmetric if $F(x,y)=F(y,x)$ for every $x,y \in \mathbb{R}^{n}$.

The kernels $K_{sp},K_{tp}:\mathbb{R}^{n}\times\mathbb{R}^{n}\rightarrow \mathbb{R}$ are measurable, symmetric functions that satisfy 
\begin{equation}\label{kernel.growth}
\frac{\nu}{|x-y|^{n+sp}} \leq K_{sp}(x,y) \leq \frac{L}{|x-y|^{n+sp}}, \qquad
\frac{\nu}{|x-y|^{n+tp}} \leq K_{tp}(x,y) \leq \frac{L}{|x-y|^{n+tp}}
\end{equation}
for a.e. $x,y \in\mathbb{R}^{n}$ with $x\neq y$, where $0 < \nu \le 1 \le L < \infty$ and
\begin{equation}\label{power}
 0 < s < t < 1 < p < \infty.
\end{equation} 
The modulating coefficients $a,b:\mathbb{R}^{n}\times\mathbb{R}^{n} \rightarrow \mathbb{R}$ are nonnegative, measurable and symmetric functions that satisfy the bound
\begin{equation}\label{a.bound}
\nu \le a(x,y) + b(x,y) \le L
\end{equation} 
for a.e. $x,y \in \ern$.

In addition, for H\"older regularity and Harnack inequality, 
we further assume the H\"older continuity of $b(\cdot,\cdot)$ as follows: There exist two numbers $[b]_{\alpha}\geq 0$ and $\alpha \in (0,1]$ such that
\begin{equation}\label{a.holder}
|b(x_{1},y_{1}) - b(x_{2},y_{2})| \leq [b]_{\alpha}(|x_{1}-x_{2}|+|y_{1}-y_{2}|)^{\alpha}
\end{equation}
for every $(x_{1},y_{1}),(x_{2},y_{2})\in\ern\times\ern$, and 
\begin{equation}\label{assumption.hol}
\quad (t-s)p \le \alpha.
\end{equation}

\begin{remark}
For any ball $B_{r}\subset \Omega$, we set
\begin{equation*}
b^{+}_{B_{r}} \coloneqq \sup_{x,y \in B_{r}}b(x,y), \qquad b^{-}_{B_{r}} \coloneqq \inf_{x,y \in B_{r}}b(x,y).
\end{equation*}
Moreover, when \eqref{a.holder} is in force, we choose a small radius $R_{0} \equiv R_{0}(\nu,\alpha,[b]_{\alpha}) \in (0,1)$ satisfying
\begin{equation}\label{R0.choice}
    [b]_{\alpha}(2R_{0})^{\alpha} \le \frac{\nu}{8}.
\end{equation}
We observe that, when $r \le R_{0}$, 
\begin{align}
b^{+}_{B_{r}} \leq \frac{\nu}{4} & \;\; \Longrightarrow \;\; \frac{3\nu}{4} \leq a(\cdot,\cdot)\leq L \ \text{ in } \ B_r\times B_r, \label{alt.dp} \\
b^{+}_{B_{r}} > \frac{\nu}{4} & \;\; \Longrightarrow \;\; 
b^{+}_{B_{r}} \overset{\eqref{a.holder}}{\le} b^{-}_{B_{r}} + [b]_{\alpha}(2r)^{\alpha} \overset{\eqref{R0.choice}}{\le} 2b^{-}_{B_{r}}. \label{alt.p}
\end{align}
Accordingly, if we set
\begin{equation}\label{G.def}
G_{B_{r}}(\tau) \coloneqq \sup_{x,y\in B_{r}}\left(a(x,y)\frac{\tau^{p}}{r^{sp}} + b(x,y)\frac{\tau^{p}}{r^{tp}}\right), \qquad g_{B_{r}}(\tau) \coloneqq \frac{G_{B_{r}}(\tau)}{\tau}, \qquad \tau \ge 0.
\end{equation}
then we have
\begin{align}\label{alt.dG}
b^{+}_{B_r} \le  \frac{\nu}{4} & \;\; \Longrightarrow \;\; G_{B_r}(\tau) \approx \frac{\tau^p}{r^{sp}} + b^+_{B_r}\frac{\tau^p}{r^{tp}}\approx \frac{\tau^p}{r^{sp}} + b^-_{B_r}\frac{\tau^p}{r^{tp}},
\end{align}
\begin{align}\label{alt.G}
b^{+}_{B_r} > \frac{\nu}{4} & \;\; \Longrightarrow \;\; G_{B_r}(\tau) \approx \frac{\tau^p}{r^{tp}}.
\end{align}
\end{remark}

We next introduce the nonlocal tail, which is one of the essential tools in analyzing local regularity for fractional equations.  
For $f:\mathbb{R}^{n}\rightarrow \mathbb{R}$, $x_{0} \in \mathbb{R}^{n}$ and $r,\rho>0$, we denote
\begin{equation}\label{def.T}
T(f;x_{0},r,\rho) \coloneqq \sup_{x \in B_{\rho}(x_{0})}\int_{\mathbb{R}^{n}\setminus B_{r}(x_{0})}\left(a(x,y)\frac{|f(y)|^{p-1}}{|y-x_{0}|^{n+sp}} + b(x,y)\frac{|f(y)|^{p-1}}{|y-x_{0}|^{n+tp}}\right)\,dy.
\end{equation}
We will omit the point $x_{0}$ if it is clear from context. 

With the definitions of weak solutions and relevant function spaces to be introduced in the next section, here we state our main results. 

\begin{theorem}[Local boundedness]\label{thm.bdd}
Let  $u \in \mathcal{A}(\Omega) \cap \mathcal{T}(\mathbb{R}^{n})$ be a weak subsolution to \eqref{main.eq} under assumptions \eqref{kernel.growth}--\eqref{a.bound}, and let $B_{r}\Subset \Omega$ be a ball with $r\le1$.
\begin{enumerate}
\item We have 
\begin{equation}\label{est:bdd1}
\esssup_{B_{r/2}} u \le  c\,r^{-(t-s)/\beta} \left(\mean{B_{r}}u_{+}^p\,dx\right)^{1/p} + \left[ r^{sp} T(u_+;r/2,r)\right]^{1/(p-1)}
\end{equation}
for a constant $c \equiv c(n,s,t,p,\nu,L)>0$, where 
\begin{equation}\label{def.beta}
\beta \coloneqq \min\{sp/n,p-1\}.
\end{equation} 
\item Assume further that $b(\cdot,\cdot)$ satisfies \eqref{a.holder} and \eqref{assumption.hol}. Then we have
\begin{equation}\label{est:bdd2}
\begin{aligned}
\esssup_{B_{r/2}} u & \le c \left(\mean{B_{r}}u_{+}^p\,dx\right)^{1/p} + \left[ \frac{1}{G_{B_r}(1)} T(u_+;r/2,r)\right]^{1/(p-1)} \\
& = c\left(\mean{B_{r}}u_{+}^{p}\,dx\right)^{1/p} + g_{B_{r}}^{-1}(T(u_{+};r/2,r))
\end{aligned}
\end{equation}
whenever $r \le R_{0}$, where $R_{0}$ is given in \eqref{R0.choice} and $c\equiv c(n,s,t,p,\nu,L,\alpha,[b]_{\alpha})>0$.  
\end{enumerate}
Consequently, if $u \in \mathcal{A}(\Omega)\cap \mathcal{T}(\mathbb{R}^{n})$ is a weak solution to \eqref{main.eq}, then $u \in L^{\infty}_{\loc}(\Omega)$ and estimate \eqref{est:bdd1} or \eqref{est:bdd2} holds with $|u|$ in each case.
\end{theorem}

\begin{theorem}[H\"older regularity]\label{thm.hol}
Let $u \in \mathcal{A}(\Omega) \cap \mathcal{T}(\mathbb{R}^{n})$ be a weak solution to \eqref{main.eq} under assumptions \eqref{kernel.growth}--\eqref{assumption.hol}. Then there exists an exponent $\gamma \equiv \gamma(n,s,t,p,\nu,L,\alpha,[b]_{\alpha}) \in (0,1)$ such that $u \in C^{0,\gamma}_{\loc}(\Omega)$. Moreover, for any ball $B_{r} \Subset \Omega$ with $r\le R_{0}$, where $R_{0}$ is given in \eqref{R0.choice}, we have
\begin{equation}\label{est:hol}
[u]_{C^{0,\gamma}(B_{r/2})} \le cr^{-\gamma}\left[\left(\mean{B_{r}}|u|^{p}\,dx\right)^{1/p} + g_{B_{r}}^{-1}(T(u;r/2,r))\right]
\end{equation}
for a constant $c\equiv c(n,s,t,p,\nu,L,\alpha,[b]_{\alpha})>0$.
\end{theorem}

\begin{theorem}[Harnack inequality]\label{thm.harnack}
Let $u \in \mathcal{A}(\Omega) \cap \mathcal{T}(\mathbb{R}^{n})$ be a weak solution to \eqref{main.eq} under assumptions \eqref{kernel.growth}--\eqref{assumption.hol} with $a(\cdot,\cdot)\equiv 1$. If $u$ is nonnegative in a ball $B_{10r} \Subset \Omega$ with $10r \leq R_{0}$, where $R_{0}$ is given in \eqref{R0.choice}, then we have
\begin{equation}\label{Harnack}
\sup_{B_{r}}u \le c\inf_{B_{r}}u + cg_{B_{r}}^{-1}(T(u_{-};10r,10r)) 
\end{equation}
for a constant $c \equiv c(n,s,t,p,\nu,L,\alpha,[b]_{\alpha})>0$.
\end{theorem}
\begin{remark}
In Theorem~\ref{thm.harnack}, the condition $a(\cdot,\cdot)\equiv 1$ can be weakened as $\inf_{\mathbb{R}^n \times \mathbb{R}^n}a(\cdot,\cdot)>0$. Such a restriction is concerned with the tail estimate given in Lemma~\ref{tail.pm} below. 
Indeed, even if $a(\cdot,\cdot)\not\equiv1$, we can prove an estimate similar to \eqref{Harnack} which involves nonlocal tails of $u$ instead of $u_{-}$. 
However, in view of the natural Harnack inequalities obtained in \cite{coz,DKP14,Kas11}, we confine ourselves to the case $a(\cdot,\cdot)\equiv1$ in Theorem~\ref{thm.harnack}. 
\end{remark}

Our proof of regularity results in Theorems~\ref{thm.bdd} and \ref{thm.hol} is based on the Moser type approach in \cite{DKP16} which employs a logarithmic estimate; the extension of this approach to the double phase setting was first presented in \cite{BOS}. 
Moreover, for the Harnack inequality in Theorem~\ref{thm.harnack}, we additionally modify and develop the techniques used in \cite{DKP14} adapted to our problem. 
We stress that, in obtaining both regularity results and Harnack inequalities for \eqref{main.eq}, these processes exhibit several substantial difficulties which did not appear in \cite{BOS}. 
First, while we only had to consider the case $sp \le n$ in \cite{BOS}, here we also treat the case $sp > n$ as well; note that in the latter case we automatically obtain the local boundedness of weak solutions via the fractional Sobolev-Morrey embedding. However, in both cases, we have to obtain explicit and precise local sup-estimates which extend the one in \cite[Theorem~1.1]{DKP16}. 
Second, since we are dealing with weak solutions which are not necessarily bounded in $\mathbb{R}^{n}$, we need a delicate analysis to handle the nonlocal tails in proceeding towards the proof of our Harnack inequality. 
More precisely, with the same spirit as in the case of fractional $p$-Laplacian \cite{coz,DKP14}, we aim to prove a Harnack inequality involving the tail of $u_{-} \coloneqq \max\{-u,0\}$ only, which reduces to the classical Harnack inequality when dealing with solutions nonnegative in the whole $\mathbb{R}^{n}$. This is indeed the most problematic issue in the case of nonlocal nonautonomous problems. In order to address this issue, we establish a weak Harnack type estimate (Lemma~\ref{weak.harnack}) and a tail estimate (Lemma~\ref{tail.pm}), and then combine them with the local sup-estimate \eqref{est:bdd2}. All of these estimates encode the long-range interactions of solutions in an optimal way.

In our proofs, we utilize assumptions \eqref{kernel.growth}--\eqref{assumption.hol} to deal with the two modulating coefficients $a(\cdot,\cdot)$ and $b(\cdot,\cdot)$ in \eqref{main.eq}. Specifically, with $B_r \subset \Omega$ being a fixed ball, we divide cases as follows: 
\begin{itemize}
\item[(i)] If $b^{+}_{B_r}\leq \nu/4$, then \eqref{alt.dp} implies that $a(x,y)$ satisfies uniform ellipticity in the sense that the first term $a(x,y)K_{sp}(x,y) \approx K_{sp}(x,y)$  in $B_r \times B_r$, while the second term $b(x,y)K_{tp}(x,y)$ has a higher order with $b(x,y)$ being possibly degenerate in $B_r \times B_r$. As a consequence, the equation \eqref{main.eq} features a nonlocal double phase structure analogous to \eqref{model.one} with $p=q$.
\item[(ii)] If $b^{+}_{B_r}>\nu/4$, then \eqref{alt.p} implies that $b(x,y)$ satisfies uniform ellipticity in $B_r \times B_r$ with ellipticity ratio $2b^-_{B_{r}}/b^-_{B_{r}}=2$. In other words, the term $b(x,y)K_{tp}(x,y) \approx K_{tp}(x,y)$ becomes a dominating term and $a(x,y)K_{sp}(x,y)$ is regarded as a lower order term in $B_{r} \times B_r$. In this case, the equation \eqref{main.eq} becomes a $(t,p)$-Laplacian type equation with an $(s,p)$-Laplacian type lower-order term.
\end{itemize}

The remaining part of this paper is organized as follows. In the next section, we introduce basic notation and function spaces which will be used throughout this paper. In Section~\ref{Caccio.bounded}, we employ a Caccioppoli type estimate to prove Theorem~\ref{thm.bdd}. After obtaining a logarithmic lemma and an expansion of positivity lemma in Section~\ref{Holder}, we finally prove Theorems~\ref{thm.hol} and \ref{thm.harnack} in Section~\ref{sec.mainthm}.

\section{Preliminaries}
\subsection{Notation and function spaces}
We denote by $c$ a generic positive constant, whose specific value may vary from line to line. We denote its dependence in parentheses when needed, using the abbreviation
\begin{equation*}
\data \coloneqq (n,s,t,p,\nu,L,\alpha, [b]_{\alpha}),
\end{equation*}
where $\alpha$ and $[b]_{\alpha}$ are given in \eqref{a.holder}. 
For numbers $A,B>0$, we write $A \lesssim B$ if $A \le cB$ holds for a constant $c>1$ depending only on data. Moreover, we write $A \approx B$ if $A \lesssim B$ and $B \lesssim A$.

As usual, $B_{r}(x_{0}) \coloneqq \{x\in\mathbb{R}^{n}:|x-x_{0}|<r\}$ is the open ball in $\mathbb{R}^{n}$ with center $x_{0} \in \mathbb{R}^{n}$ and radius $r>0$. We omit the center when it is clear in the context. Moreover, unless otherwise stated, different balls in the same context are concentric. 

For a real-valued function $f$, we write $f_{\pm} \coloneqq \max\{\pm f,0\}$. If $f$ is integrable over a measurable set $U$ with $0<|U|<\infty$, we denote its integral average over $U$ by
\begin{equation*}
(f)_{U} \coloneqq \mean{U}f\,dx \coloneqq \frac{1}{|U|}\int_{U}f\,dx.
\end{equation*}

For any open set $U\subseteq \mathbb{R}^{n}$, $s\in(0,1)$ and $p\geq 1$, the fractional Sobolev space $W^{s,p}(U)$ is the set of all functions $f \in L^{p}(U)$ satisfying
\begin{align*}
\lVert f \rVert_{W^{s,p}(U)} 
 \coloneqq \left(\int_{U}|f|^{p}\,dx\right)^{1/p} + \left(\int_{U}\int_{U}\frac{|f(x)-f(y)|^{p}}{|x-y|^{n+sp}}\,dx\,dy\right)^{1/p} < \infty.
\end{align*}
We always assume that $s$, $t$, and $p$ satisfy \eqref{power} and that $K_{sp},K_{tp},a,b:\R^n\times \R^n \rightarrow \mathbb{R}$ satisfy \eqref{kernel.growth} and \eqref{a.bound}. 
Recalling \eqref{COmega}, we define the admissible set $\mathcal{A}(\Omega)$ of \eqref{functional} by saying that $f \in \mathcal{A}(\Omega)$ if and only if 
\begin{equation*}
 f|_{\Omega}\in L^p(\Omega)\ \text{ and } \  \iint_{\mathcal{C}_{\Omega}}\left(a(x,y)\frac{|f(x)-f(y)|^{p}}{|x-y|^{n+sp}} + b(x,y)\frac{|f(x)-f(y)|^{p}}{|x-y|^{n+tp}}\right)\,dx\,dy < \infty,
\end{equation*}
where $\mathcal{C}_{\Omega}$ is defined in \eqref{COmega}.
In particular, by \eqref{a.bound} and Lemma~\ref{inclusion} below, we have
\begin{equation*}
f \in \mathcal{A}(\Omega) \;\; \Longrightarrow \;\; f|_{\Omega} \in W^{s,p}(\Omega). 
\end{equation*}
Accordingly, we say that a function $u \in \mathcal{A}(\Omega)$ is a weak solution to \eqref{main.eq} if
\begin{equation}\label{weak.formulation}
\begin{aligned}
& \iint_{\mathcal{C}_{\Omega}}a(x,y)|u(x)-u(y)|^{p-2}(u(x)-u(y))(\varphi(x)-\varphi(y))K_{sp}(x,y)\,dx\,dy \\
&\quad + \iint_{\mathcal{C}_{\Omega}}b(x,y)|u(x)-u(y)|^{p-2}(u(x)-u(y))(\varphi(x)-\varphi(y))K_{tp}(x,y)\,dx\,dy = 0
\end{aligned}
\end{equation}
holds for every $\varphi \in \mathcal{A}(\Omega)$ with $\varphi = 0$ a.e. in $\mathbb{R}^{n}\setminus\Omega$.  In a similar way, we say that $u \in \mathcal{A}(\Omega)$ is a weak subsolution (resp. supersolution) if \eqref{weak.formulation} with ``='' replaced by ``$\leq$ (resp. $\geq$)'' holds for every $\varphi \in \mathcal{A}(\Omega)$ satisfying $\varphi \geq 0$ a.e. in $\mathbb{R}^{n}$ and $\varphi = 0$ a.e. in $\mathbb{R}^{n}\setminus\Omega$. 
The existence of weak solutions to \eqref{main.eq} (coupled with a suitable Dirichlet boundary condition) can be proved via direct methods in the calculus of variations, see \cite[Section~3]{BOS}.

Recalling the definition of nonlocal tail in \eqref{def.T}, we also consider the tail space
\[ \mathcal{T}(\mathbb{R}^{n}) \coloneqq \left\{f \in L^{p-1}_{\loc}(\mathbb{R}^{n}): T(f;x_{0},r,\rho) \text{ is finite for any }x_{0}\in\mathbb{R}^{n},r>0,\rho>0 \right\}. \]
It is not difficult to see that our tail space includes the standard one: 
\begin{align*}
L^{p-1}_{sp}(\mathbb{R}^{n}) \coloneqq \left\{f \in L^{p-1}_{\loc}(\mathbb{R}^{n}): \int_{\mathbb{R}^{n}}\frac{|f(x)|^{p-1}}{(1+|x|)^{n+sp}}\,dx < \infty  \right\} \subset \mathcal{T}(\mathbb{R}^n),
\end{align*} 
see for instance \cite[Section~2.1]{BOS}. In particular, if $f \in L^{q}(\mathbb{R}^{n})$ for some $q>p-1$, or if $f \in L^{p-1}(B_{R})\cap L^{\infty}(\mathbb{R}^{n}\setminus B_{R})$ for a ball $B_{R}$, then $f \in \mathcal{T}(\mathbb{R}^{n})$.

\subsection{Technical results}
We recall several inequalities concerning fractional Sobolev functions.
For more on fractional Sobolev spaces, we refer to \cite{DPV}.

\begin{lemma}[{\cite[Lemma~2.2]{BOS}}]\label{inclusion}
Assume that $s$, $t$ and $p$ satisfy \eqref{power}. If $f \in W^{t,p}(U)$ for a bounded open set $U \subset \mathbb{R}^{n}$, then 
\begin{equation*}
\left(\int_{U}\int_{U}\frac{|f(x)-f(y)|^{p}}{|x-y|^{n+sp}}\,dx\,dy\right)^{1/p} \leq c(\diam(U))^{t-s}\left(\int_{U}\int_{U}\frac{|f(x)-f(y)|^p}{|x-y|^{n+tp}}\,dx\,dy\right)^{1/p}
\end{equation*}
holds for a constant $c\equiv c(n,s,t,p)>0$.
\end{lemma}

\begin{lemma}[{\cite[Lemma 2.5]{Ok23}}]\label{SP}
Let $s\in(0,1)$ and $p\geq 1$; denote the $s$-fractional Sobolev conjugate of $p$ by
\begin{equation*}
p_{s}^{*} \coloneqq 
\begin{cases}
np/(n-sp) & \text{when }\ sp < n, \\
\textrm{any number in }(p,\infty) & \text{when }\ sp \ge n.
\end{cases}
\end{equation*}  
Then we have
\begin{equation*}
\left(\mean{B_{r}}|f-(f)_{B_{r}}|^{p_{s}^{*}}\,dx\right)^{p/p_{s}^{*}} \leq cr^{sp}\mean{B_{r}}\int_{B_{r}}\frac{|f(x)-f(y)|^{p}}{|x-y|^{n+sp}}\,dx\,dy
\end{equation*}
for any $f \in W^{s,p}(B_{r})$, where $c\equiv c(n,s,p)>0$.
\end{lemma}

The following lemma can be proved in the same way as in \cite[Lemma~2.4]{BOS} and \cite[Lemma~2.2]{BLS}.
\begin{lemma}\label{ineq1}
Assume that $s$, $t$ and $p$ satisfy \eqref{power}. Then for any $f \in L^{p}(B_{r})$ and any $a_{0},b_{0} \ge 0$, we have
\begin{equation*}
\begin{aligned}
&\mean{B_{r}}\left(a_{0}\frac{|f|^{p}}{r^{sp}} + b_{0}\frac{|f|^{p}}{r^{tp}}\right)\,dx \\
&\leq  c\left(\frac{|\supp\, f|}{|B_{r}|}\right)^{sp/n}\mean{B_{r}}\int_{B_{r}}\left(a_0 \frac{|f(x)-f(y)|^{p}}{|x-y|^{n+sp}}+b_0 \frac{|f(x)-f(y)|^{p}}{|x-y|^{n+tp}}\right)\,dx\,dy \\
&\quad + c\left(\frac{|\supp\, f|}{|B_{r}|}\right)^{p-1}\mean{B_{r}}\left(a_{0}\frac{|f|^{p}}{r^{sp}} + b_{0}\frac{|f|^{p}}{r^{tp}}\right)\,dx\\
\end{aligned}
\end{equation*}
whenever the right-hand side is finite, where
$c\equiv c(n,s,t,p)>0$.
\end{lemma}
\begin{proof}
Using H\"older's inequality and Lemma~\ref{SP}, we have
\begin{equation*}
\begin{aligned}
\mean{B_{r}}\frac{|f|^{p}}{r^{sp}}\,dx & \le c\left(\frac{|\supp f|}{|B_{r}|}\right)^{sp/n}\left(\mean{B_{r}}\left|\frac{f-(f)_{B_{r}}}{r^{s}}\right|^{p^{*}_{s}}\,dx\right)^{p/p^{*}_{s}} + c\frac{|(f)_{B_{r}}|^{p}}{r^{sp}} \\
& \le c\left(\frac{|\supp f|}{|B_{r}|}\right)^{sp/n}\mean{B_{r}}\int_{B_{r}}\frac{|f(x)-f(y)|^{p}}{|x-y|^{n+sp}}\,dx\,dy + c\left(\frac{|\supp f|}{|B_{r}|}\right)^{p-1}\mean{B_{r}}\frac{|f|^{p}}{r^{sp}}\,dx,
\end{aligned}
\end{equation*}
whenever the right-hand side is finite; the same is true when $s$ is replaced by $t$. 
\end{proof}

Finally, we mention standard iteration lemmas.
\begin{lemma}[{\cite[Lemma 7.1]{MR1962933}}]\label{iterlem} 
Let $\{y_{i}\}_{i=0}^{\infty}$ be a sequence of nonnegative numbers satisfying
\begin{equation*}
y_{i+1} \leq b_{1}b_{2}^{i}y_{i}^{1+\beta}, \qquad i=0,1,2,... 
\end{equation*}
for some constants $b_{1},\beta>0$ and $b_{2}>1$. If
\begin{equation*}
y_{0} \leq b_{1}^{-1/\beta}b_{2}^{-1/\beta^{2}},
\end{equation*}
then $y_{i} \rightarrow 0$ as $i\rightarrow\infty$.
\end{lemma}

\begin{lemma}[{\cite[Lemma 6.1]{MR1962933}}]\label{techlem}
Let $Z:[r_{0},r_{1}] \rightarrow \mathbb{R}$ be a nonnegative and bounded function; let $\vartheta \in (0,1)$ and $C_{1},C_{2},\chi > 0$ be numbers. Assume that
\begin{equation*}
Z(\varrho_{0}) \le \vartheta Z(\varrho_{1}) + \frac{C_{1}}{(\varrho_{1}-\varrho_{0})^{\chi}} + C_{2}
\end{equation*}
holds for every choice of $\varrho_{0}$ and $\varrho_{1}$ such that $r_{0} \le \varrho_{0} < \varrho_{1} \le r_{1}$. Then the following inequality holds with $c\equiv c(\vartheta,\chi)>0$:
\begin{equation*}
Z(r_{0}) \le c\left[\frac{C_{1}}{(r_{1}-r_{0})^{\chi}} + C_{2}\right].
\end{equation*}
\end{lemma}

\section{Proof of Theorem~\ref{thm.bdd}}\label{Caccio.bounded}
The following Caccioppoli type estimate can be proved in a standard way, see for instance \cite[Lemma~4.2]{BOS}.

\begin{lemma}\label{caccioppoli2}
Let $u\in\mathcal{A}(\Omega) \cap \mathcal{T}(\mathbb{R}^{n})$ be a weak solution to \eqref{main.eq} under assumptions \eqref{kernel.growth}--\eqref{a.bound}. 
There exists a constant $c\equiv c(n,s,t,p,\nu,L)>0$ such that
\begin{equation}\label{caccio.2}
\begin{aligned}
& \int_{B_{\rho}}\int_{B_{\rho}}\left(a(x,y)\frac{|w_{\pm}(x)-w_{\pm}(y)|^{p}}{|x-y|^{n+sp}} + b(x,y)\frac{|w_{\pm}(x)-w_{\pm}(y)|^{p}}{|x-y|^{n+tp}}\right)\,dx\,dy  \\
& \le \frac{c}{(r-\rho)^{p}}\int_{B_{r}}\int_{B_{r}}\left(a(x,y)\frac{w_{\pm}^{p}(x)}{|x-y|^{n+(s-1)p}}+ b(x,y)\frac{w_{\pm}^{p}(x)}{|x-y|^{n+(t-1)p}}\right)\,dx\,dy \\
& \quad\; + c\left(\frac{r}{r-\rho}\right)^{n+tp}T(w_{\pm};r,r)\int_{B_{r}}w_{\pm}\,dx
\end{aligned}
\end{equation}
holds whenever $B_{\rho} \subset  B_{r} \Subset \Omega$ are concentric balls, where $w_{\pm} \coloneqq (u-k)_{\pm}$ for any $k\in\mathbb{R}$.
Moreover, estimate \eqref{caccio.2} continues to hold for $w_{+}$ (resp. $w_{-}$) when $u$ is merely a weak subsolution (resp. supersolution) to \eqref{main.eq}.
\end{lemma}

We now prove Theorem~\ref{thm.bdd}.
\begin{proof}[Proof of Theorem \ref{thm.bdd}]
Let $B_{r} \equiv B_{r}(x_{0}) \Subset \Omega$ be a fixed ball with $r \leq 1$. For $r/2 \leq \rho < \sigma \leq r$ and $k>0$, we denote \begin{equation*}
A^{+}(k,\rho) \coloneqq \{x\in B_{\rho}: u(x) \geq k \}.
\end{equation*} 
Note that  for $0<h<k$ and $x\in A^{+}(k,\rho) \subset A^{+}(h,\rho)$,
\begin{align*}
&(u(x)-h)_{+} = u(x)-h \geq k-h, \\
&(u(x)-h)_{+} = u(x)-h \geq u(x)-k = (u(x)-k)_{+}.
\end{align*}
Thus, we have 
\begin{equation}\label{aplus.est}
|A^{+}(k,\rho)| \leq \int_{A^{+}(k,\rho)}\frac{(u-h)_{+}^{p}}{(k-h)^{p}}\,dx \leq \frac{1}{(k-h)^{p}}\int_{A^{+}(h,\sigma)}(u-h)_{+}^p\,dx
\end{equation}
and
\begin{equation}\label{uk.est}
\int_{B_{\sigma}}(u-k)_{+}\,dx \leq \int_{B_{\sigma}}(u-h)_{+}\left(\frac{(u-h)_{+}}{k-h}\right)^{p-1}\,dx  \leq \frac{1}{(k-h)^{p-1}}\int_{B_{\sigma}}(u-h)_{+}^p\,dx.
\end{equation}

\textit{Step 1: Proof of \eqref{est:bdd1}.} We first prove \eqref{est:bdd1} in the case that $b(\cdot,\cdot)$ is merely measurable. 
Applying Lemma \ref{ineq1} with $f \equiv (u-k)_{+}$, $a_0=1$ and $b_0=0$, we get
\begin{equation}\label{sf}
\begin{aligned}
& \mean{B_{\rho}}(u-k)_{+}^p\,dx \\
&  \leq c \left(\frac{|A^{+}(k,\rho)|}{|B_{\rho}|}\right)^{\beta}\left(\rho^{sp}\mean{B_{\rho}}\int_{B_{\rho}}\frac{|(u(x)-k)_{+}-(u(y)-k)_{+}|^{p}}{|x-y|^{n+sp}}\,dx\,dy +\mean{B_{\sigma}} (u-k)_{+}^{p} \,dx\right),
\end{aligned}
\end{equation}
where $\beta$ is given in \eqref{def.beta}. 
Observe that, since $r \le 1$, 
\[
1 \le \frac{1}{\nu} \big(a(x,y)+b(x,y)\big) \le \frac{1}{\nu} \left(a(x,y)+ b(x,y)\frac{2^{(t-s)p}}{|x-y|^{(t-s)p}}\right) \qquad \text{for any } x,y\in B_r.
\]
Therefore, an application of Lemma~\ref{caccioppoli2} gives
\begin{equation*}
\begin{aligned}
& \mean{B_{\rho}}\int_{B_{\rho}}\frac{|(u(x)-k)_+-(u(y)-k)_+|^{p}}{|x-y|^{n+sp}}\,dx\,dy \\
& \leq \frac{c}{(\sigma-\rho)^{p}}\mean{B_{\sigma}}(u(x)-h)_{+}^{p}\int_{B_{\sigma}}\left(\frac{1}{|x-y|^{n+(s-1)p}}+\frac{1}{|x-y|^{n+(t-1)p}}\right)\,dy\,dx \\
& \quad\; +c\left(\frac{\sigma}{\sigma-\rho}\right)^{n+tp}T((u-k)_{+};\sigma,\sigma)\mean{B_{\sigma}}(u-k)_{+}\,dx \\
& \leq \frac{cr^{(1-t)p}}{(\sigma-\rho)^{p}}\mean{B_{\sigma}}(u-h)_{+}^p\,dx + \frac{cr^{n+tp}}{(\sigma-\rho)^{n+tp}}T((u-k)_{+};r/2,r)\mean{B_{\sigma}}(u-k)_{+}\,dx.
\end{aligned}
\end{equation*}
Combining this estimate together with \eqref{sf}, \eqref{aplus.est} and \eqref{uk.est}, we arrive at
\begin{equation}\label{cacciokh}
\begin{aligned}
\mean{B_{\rho}}(u-k)_{+}^p\,dx
& \leq c \left(\mean{B_{\sigma}}\left[\frac{(u-h)_{+}}{k-h}\right]^p\,dx\right)^{\beta}  \mean{B_{\sigma}}(u-h)_{+}^p\,dx \\
& \quad\; \cdot \left[\frac{r^{p+(s-t)p}}{(\sigma-\rho)^{p}}+ \frac{r^{n+tp}}{(\sigma-\rho)^{n+tp}}\frac{r^{sp}}{(k-h)^{p-1}}T((u-k)_{+};r/2,r)+1\right].
\end{aligned}
\end{equation}

Now, for $i=0,1,2,\ldots$ and $k_0>0$ with
\begin{equation}\label{choosek01}
k_{0} \ge \left[ r^{sp} T(u_+;r/2,r)\right]^{1/(p-1)},
\end{equation}
define 
\begin{equation*}
\sigma_{i} \coloneqq \frac{r}{2}(1+2^{-i})
\quad \textrm{and}\quad
 k_{i} \coloneqq 2k_{0}(1-2^{-i-1}).
\end{equation*}
We choose $k=k_{i+1}$, $h=k_i$, $\rho=\sigma_{i+1}$, and $\sigma=\sigma_{i}$ in \eqref{cacciokh}. Then, dividing both sides of the resulting inequality by $k_0^p$, we arrive at
\[
\mean{B_{\sigma_{i+1}}}\left[\frac{(u-k_{i+1})_{+}}{k_{0}}\right]^p\,dx
 \leq c 2^{i(\beta p+n+tp+p-1)}r^{(s-t)p}\left(\mean{B_{\sigma_i}}\left[\frac{(u-k_i)_{+}}{k_0}\right]^p\,dx\right)^{1+\beta}. 
\]
Setting
\begin{equation*}
y_{i} \coloneqq \frac{1}{|B_r|}\int_{A^{+}(k_{i},\sigma_{i})}\left[\frac{(u-k_{i})_{+}}{k_0}\right]^p\,dx, 
\quad i=0,1,2,\dots,
\end{equation*}
the above inequality reads as
\begin{align*}
y_{i+1}  \leq \tilde{c}r^{(s-t)p}2^{i (\beta p +n+tp+p-1)}y_{i}^{1+\beta}
\end{align*}
for some $\tilde{c}\equiv \tilde{c}(n,s,t,p,\nu,L)>0$. Therefore if
\begin{equation}\label{choosek02}
y_0=\frac{1}{k_0^p} \mean{B_{r}}(u-k_{0})_{+}^p\,dx  \le  \tilde{c}^{-1/\beta} r^{(t-s)p/\beta} 2^{-(\beta p +n+tp+p-1)/\beta^2},
\end{equation}
we have that $y_i\to 0$ as $i\to \infty$ by Lemma \ref{iterlem}, so together with $\lim_{i\rightarrow\infty}\sigma_{i}=r/2$ we have
\[
u \le 2k_0  \quad \text{a.e. in }\ B_{r/2}.
\]
Note that the conditions \eqref{choosek01} and \eqref{choosek02} of $k_0$ are satisfied if we choose $k_0>0$ such that  
\[
k_0 = c r^{-(t-s)/\beta} \left(\mean{B_{r}}u_{+}^p\,dx\right)^{1/p} + \left[ r^{sp} T(u_+;r/2,r)\right]^{1/(p-1)}
\]
for some large constant $c\equiv c(n,s,t,p,\nu,L)>0$. Therefore, we obtain \eqref{est:bdd1}.

We now assume that $b(\cdot,\cdot)$ satisfies \eqref{a.holder} and \eqref{assumption.hol}. To prove \eqref{est:bdd2}, we distinguish two cases. 

\textit{Step 2: Proof of \eqref{est:bdd2} in the case \eqref{alt.p}.} 
Assume that $r \leq R_0$.  
We apply Lemma \ref{ineq1} with $f \equiv (u-k)_{+}$, $a_0=0$ and $b_0=1$ to get
\begin{equation}\label{sf1}
\begin{aligned}
 \mean{B_{\rho}}(u-k)_{+}^p\,dx
&  \leq c \rho^{tp}\left(\frac{|A^{+}(k,\rho)|}{|B_{\rho}|}\right)^{sp/n}\mean{B_{\rho}}\int_{B_{\rho}}\frac{|(u(x)-k)_{+}-(u(y)-k)_{+}|^{p}}{|x-y|^{n+tp}}\,dx\,dy \\
& \quad + c \left(\frac{|A^{+}(k,\rho)|}{|B_{\rho}|}\right)^{p-1}\mean{B_{\sigma}} (u-k)_{+}^{p} \,dx.
\end{aligned}
\end{equation}
Since $b(x,y)\ge \nu/4$ for $x,y\in B_r$ by \eqref{alt.p}, applying Lemma~\ref{caccioppoli2} gives
\begin{equation*}
\begin{aligned}
& \mean{B_{\rho}}\int_{B_{\rho}}\frac{|(u(x)-k)_+-(u(y)-k)_+|^{p}}{|x-y|^{n+tp}}\,dx\,dy \\
& \leq \frac{cr^{(1-t)p}}{(\sigma-\rho)^{p}}\mean{B_{\sigma}}(u-h)_{+}^p\,dx + \frac{cr^{n+tp}}{(\sigma-\rho)^{n+tp}}T((u-k)_{+};r/2,r)\mean{B_{\sigma}}(u-k)_{+}\,dx.
\end{aligned}
\end{equation*}
Combining this estimate together with \eqref{sf1}, \eqref{aplus.est}, \eqref{uk.est}, and recalling $G_{B_{r}}(1)\approx r^{-tp}$ from \eqref{alt.dG} imply 
\begin{equation}\label{cacciokh1}
\begin{aligned}
\mean{B_{\rho}}(u-k)_{+}^p\,dx
& \leq c \left(\mean{B_{\sigma}}\left[\frac{(u-h)_{+}}{k-h}\right]^p\,dx\right)^{sp/n}  \mean{B_{\sigma}}(u-h)_{+}^p\,dx \\
& \quad\; \cdot \left[\frac{r^{p}}{(\sigma-\rho)^{p}}+ \frac{r^{n+tp}}{(\sigma-\rho)^{n+tp}}\frac{1}{(k-h)^{p-1}G_{B_r}(1)}T((u-k)_{+};r/2,r)\right] \\
& \quad + \frac{c}{(k-h)^{p(p-1)}}\left(\mean{B_{\sigma}}(u-h)_{+}^p\,dx\right)^{p}.
\end{aligned}
\end{equation}
Note that the estimate \eqref{cacciokh1} is similar to \eqref{cacciokh}, except that the terms  $r^{p+(s-t)p}$ and $r^{sp}$ are replaced by $r^{p}$ and $1/G_{B_r}(1)$, respectively. The remaining part of the proof is exactly the same as in \textit{Step~1}, hence one can obtain the estimate \eqref{est:bdd2}.

\textit{Step 3: Proof of \eqref{est:bdd2} in the case \eqref{alt.dp}.} Here we recall the function $G_{B_r}(\cdot)$ defined in \eqref{G.def}; in the following, we simply denote $G=G_{B_r}$. Observe that
\begin{equation}\label{aplus.estG}
\int_{B_{\sigma}}(u-k)_{+}\,dx
 \leq \frac{k-h}{G(k-h)}\int_{B_{\sigma}}G((u-h)_{+})\,dx
 \leq \frac{c}{(k-h)^{p-1}G(1)}\int_{B_{\sigma}}G((u-h)_{+})\,dx.
\end{equation}
Apply Lemma \ref{ineq1} with $f \equiv (u-k)_{+}$, $a_0=1$ and $b_0=b_{B_r}^-$ to get
\begin{equation}\label{sf1G}
\begin{aligned}
 & \mean{B_{\rho}}G((u-k)_{+})\,dx \le  c \mean{B_{\rho}} \frac{(u-k)_{+}^p}{r^{sp}}+ b^-_{B_r} \frac{(u-k)_{+}^p}{r^{tp}} \,dx  \\
&  \leq c \left(\frac{|A^{+}(k,\rho)|}{|B_{\rho}|}\right)^{\frac{sp}{n}}\mean{B_{\rho}}\int_{B_{\rho}} \frac{|(u(x)-k)_{+}-(u(y)-k)_{+}|^{p}}{|x-y|^{n+sp}}+ b_{B_r}^- \frac{|(u(x)-k)_{+}-(u(y)-k)_{+}|^{p}}{|x-y|^{n+tp}}\,dx\,dy \\
& \quad\; + c \left(\frac{|A^{+}(k,\rho)|}{|B_{\rho}|}\right)^{p-1}\mean{B_{\rho}} G((u-k)_{+}) \,dx.
\end{aligned}
\end{equation}
Since $a(x,y)\ge 3\nu/4$ for $x,y\in B_r$ by \eqref{alt.dp}, applying Lemma~\ref{caccioppoli2} gives
\begin{equation*}
\begin{aligned}
&\mean{B_{\rho}}\int_{B_{\rho}}\left( \frac{|(u(x)-k)_{+}-(u(y)-k)_{+}|^{p}}{|x-y|^{n+tp}}+ b_{B_r}^- \frac{|(u(x)-k)_{+}-(u(y)-k)_{+}|^{p}}{|x-y|^{n+tp}} \right)\,dx\,dy \\
& \leq \frac{c}{(\sigma-\rho)^{p}}\mean{B_{\sigma}}(u(x)-h)_{+}^{p}\int_{B_{\sigma}}\left(\frac{a(x,y)}{|x-y|^{n+(s-1)p}}+\frac{b(x,y)}{|x-y|^{n+(t-1)p}}\right)\,dy\,dx \\
& \quad\; +c\left(\frac{\sigma}{\sigma-\rho}\right)^{n+tp}T\left((u-k)_{+};\sigma,\sigma\right)\mean{B_{\sigma}}(u-k)_{+}\,dx \\
& \leq \frac{c}{(\sigma-\rho)^{p}}\mean{B_{\sigma}}(u(x)-h)_{+}^{p}\int_{B_{\sigma}}\left(\frac{1}{|x-y|^{n+(s-1)p}}+\frac{b^+_{B_r}}{|x-y|^{n+(t-1)p}}\right)\,dy\,dx \\
& \quad\; +c\left(\frac{\sigma}{\sigma-\rho}\right)^{n+tp}T\left((u-k)_{+};\sigma,\sigma\right)\mean{B_{\sigma}}(u-k)_{+}\,dx \\
& \leq \frac{cr^{p}}{(\sigma-\rho)^{p}}\mean{B_{\sigma}}G((u-h)_{+})\,dx + c\left(\frac{r}{\sigma-\rho}\right)^{n+tp}T((u-k)_{+};r/2,r)\mean{B_{\sigma}}(u-k)_{+}\,dx.
\end{aligned}
\end{equation*}
Combining this estimate together with \eqref{aplus.est}, \eqref{aplus.estG}, \eqref{sf1G}, and recalling $G_{B_{r}}(1)\approx r^{-sp}+b^+_{B_{r}}r^{-tp}$ from \eqref{alt.G} imply
\begin{equation*}
\begin{aligned}
\mean{B_{\rho}}G((u-k)_{+})\,dx
& \leq c \left(\mean{B_{\sigma}} \left[\frac{(u-h)_{+}}{k-h}\right]^p\,dx\right)^{sp/n}  \mean{B_{\sigma}}G((u-h)_{+})\,dx \\
& \quad \;\cdot \left[\frac{r^{p}}{(\sigma-\rho)^{p}}+ \frac{r^{n+tp}}{(\sigma-\rho)^{n+tp}}\frac{1}{(k-h)^{p-1}G(1)}T((u-k)_{+};r/2,r)\right] \\
& \quad + c \left(\mean{B_{\sigma}} \left[\frac{(u-h)_{+}}{k-h}\right]^p\,dx\right)^{p-1}  \mean{B_{\sigma}}G((u-h)_{+})\,dx.
\end{aligned}
\end{equation*}
Dividing both sides of the above inequality by $r^{-sp}+b^{+}_{B_r}r^{-tp}$, we obtain the estimate \eqref{cacciokh} with $r^{p+(s-t)p}$ and $r^{sp}$ replaced by $r^{p}$ and $1/G(1)$, respectively. Hence we can conclude with the desired estimate \eqref{est:bdd2}, and the proof is complete. 
\end{proof}

\section{Expansion of positivity}\label{Holder}

Throughout this section, we assume that the modulating coefficient $b(\cdot,\cdot)$ satisfies \eqref{a.holder} with \eqref{assumption.hol}. 
We start this section with the following logarithmic type estimate, whose proof is analogous to those of \cite[Corollary~3.2]{DKP16} and \cite[Corollary~5.2]{BOS}.

\begin{lemma}\label{log.lemma}
Let $u \in \mathcal{A}(\Omega) \cap \mathcal{T}(\mathbb{R}^{n})$ be a weak supersolution to \eqref{main.eq} under assumptions \eqref{kernel.growth}--\eqref{assumption.hol}, which is nonnegative in a ball $B_{R} \equiv B_{R}(x_{0})$ with $R \le R_{0}$. For any $d,\zeta > 0$ and $\xi>1$, define
\begin{equation*}
v \coloneqq \min\{(\log(\zeta+d)-\log(u+d))_{+}, \log\xi\}.
\end{equation*} 
Then for any $r \in (0, R/2]$ and $d>0$, we have
\begin{equation}\label{log.est}
\mean{B_{r}}|v-(v)_{B_{r}}|\,dx
\leq c + cd^{1-p}\frac{1}{G_{B_{2r}}(1)}T(u_{-};R,2r)
\end{equation}
for a constant $c \equiv c(\data)>0$.
\end{lemma}
\begin{proof}
In the case $b^{+}_{B_{R}} \le \nu/4$, since $\nu/4 \le a(x,y) \le L$, the estimate is proved in \cite[Lemma~5.1]{BOS} with $q=p$. Note that, in the setting of the present paper, \cite[(5.10)]{BOS} can be replaced by
\begin{equation*}
\begin{aligned}
& \int_{B_{3r/2}}\int_{\mathbb{R}^{n}\setminus B_{R}}\frac{1}{g_{B_{2r}}(u(x)+d)}\left(a(x,y)\frac{(u_{-}(y))^{p-1}}{|x-y|^{n+sp}}+b(x,y)\frac{(u_{-}(y))^{p-1}}{|x-y|^{n+tp}}\right)\,dy\,dx \\
& \le \frac{cr^{n}}{g_{B_{2r}}(d)}\sup_{x\in B_{2r}}\int_{\mathbb{R}^{n}\setminus B_{R}}\left(a(x,y)\frac{(u_{-}(y))^{p-1}}{|y-x_{0}|^{n+sp}}+b(x,y)\frac{(u_{-}(y))^{p-1}}{|y-x_{0}|^{n+tp}}\right)\,dy \\
& = \frac{cr^{n}d^{1-p}}{G_{B_{2r}}(1)}T(u_{-};R,2r).
\end{aligned}
\end{equation*}
We now consider the case $b^{+}_{B_{R}}>\nu/4$.
For brevity, we denote
\begin{equation}\label{phip}
\Phi_{p}(\tau) \coloneqq |\tau|^{p-2}\tau , \quad \ \tau \in \mathbb{R}. 
\end{equation}
Choose a cut-off function $\phi \in C^{\infty}_{0}(B_{3r/2})$ such that $0 \le \phi \le 1$, $\phi \equiv 1$ in $B_{r}$ and $|D\phi| \le 4/r$ in $B_{r}$. Testing \eqref{weak.formulation} with $\varphi \equiv (u+d)^{1-p}\phi^{p}$, we have
\begin{equation*}
\begin{aligned}
0 & \leq \int_{B_{2r}}\int_{B_{2r}}\left[a(x,y)K_{sp}(x,y)+b(x,y)K_{tp}(x,y)\right]\Phi_{p}(u(x)-u(y))(\varphi(x)-\varphi(y))\,dx\,dy \\
& \quad\; +2\int_{\mathbb{R}^n\setminus B_{2r}}\int_{B_{2r}}\left[a(x,y)K_{sp}(x,y) + b(x,y)K_{tp}(x,y)\right]\Phi_{p}(u(x)-u(y))(\varphi(x)-\varphi(y))\,dx\,dy \\
& \eqqcolon I_{1} + I_{2}.
\end{aligned}
\end{equation*}
Estimating in the same way as in \cite[Lemma~1.3]{DKP16}, we have
\begin{equation*}
\begin{aligned}
    I_{1} & \le -\frac{1}{c}\int_{B_{2r}}\int_{B_{2r}}\left[a(x,y)K_{sp}(x,y)+b(x,y)K_{tp}(x,y)\right]\left|\log\left(\frac{u(x)+d}{u(y)+d}\right)\right|^{p}\phi^{p}(y)\,dx\,dy \\
    & \quad\; + c\int_{B_{2r}}\int_{B_{2r}}\left[a(x,y)K_{sp}(x,y)+b(x,y)K_{tp}(x,y)\right]|\phi(x)-\phi(y)|^{p}\,dx\,dy \\
    & \le -\frac{1}{c}\int_{B_{r}}\int_{B_{r}}\left|\log\left(\frac{u(x)+d}{u(y)+d}\right)\right|^{p}\frac{dx\,dy}{|x-y|^{n+tp}} + cr^{n-tp}
\end{aligned}
\end{equation*}
and
\begin{equation*}
\begin{aligned}
I_{2} & \le c\int_{\mathbb{R}^{n}\setminus B_{2r}}\int_{B_{2r}}\left[a(x,y)K_{sp}(x,y)+b(x,y)K_{tp}(x,y)\right]\phi^{p}(x)\,dx\,dy \\
& \quad\; + cd^{1-p}\int_{\mathbb{R}^{n}\setminus B_{R}}\int_{B_{2r}}\left[a(x,y)K_{sp}(x,y)+b(x,y)K_{tp}(x,y)\right]u_{-}^{p-1}(y)\,dx\,dy \\
& \le cr^{n-tp} + cd^{1-p}r^{n}T(u_{-};R,2r).
\end{aligned}
\end{equation*}
Combining the above three displays, we get
\begin{equation*}
 \int_{B_{r}}\int_{B_{r}}\frac{|\log(u(x)+d)-\log(u(y)+d)|^{p}}{|x-y|^{n+tp}}\,dx\,dy  \le cr^{n-tp} + cd^{1-p}r^{n}T(u_{-};R,2r).
\end{equation*}
In turn, an application of the fractional Poincar\'e inequality gives 
\begin{equation*} 
\begin{aligned}
\mean{B_{r}}|v-(v)_{B_{r}}|\,dx & \le \mean{B_{r}}|v-(v)_{B_{r}}|^{p}\,dx + 1 \\
& \le c + cd^{1-p}r^{tp}T(u_{-};R,2r) \le c + \frac{cd^{1-p}}{G_{B_{2r}}(1)}T(u_{-};R,2r),
\end{aligned}
\end{equation*}
which is \eqref{log.est}. The proof is complete.
\end{proof}

Using the above lemma, we prove the following result concerning expansion of positivity.

\begin{lemma}\label{density.improve}
Let $u \in \mathcal{A}(\Omega) \cap \mathcal{T}(\mathbb{R}^{n})$ be a weak supersolution to \eqref{main.eq} under assumptions \eqref{kernel.growth}--\eqref{assumption.hol}, which is nonnegative in a ball $B_{R} \equiv B_{R}(x_{0}) \Subset \Omega$ with $R \le R_{0}$. Let $k > 0$ and assume that there exists $\sigma \in (0,1]$ satisfying
\begin{equation}\label{density.assumption}
|B_{2r} \cap \{ u \ge k \}| \ge \sigma|B_{2r}| \qquad \text{for some} \quad r \in (0,R/4].
\end{equation}
Then there exists a constant $\delta \equiv \delta(\data,\sigma) \in (0,1/4)$ such that if
\begin{equation}\label{T.assumption}
d \coloneqq \left[\frac{1}{G_{B_{2r}}(1)}T(u_{-};R,2r) \right]^{1/(p-1)} = g_{B_{2r}}^{-1}(T(u_{-};R,2r)) \le \delta k,
\end{equation}
then
\begin{equation}\label{lower.bound}
\essinf_{B_{r}}u \ge \delta k.
\end{equation}
\end{lemma}
\begin{proof}
We divide the proof into two steps.

\textit{Step 1: A density estimate.} We first show that
\begin{equation}\label{density.est}
\frac{\left|B_{2r} \cap \left\{u \le 2\delta k \right\}\right|}{|B_{2r}|} \le \frac{\bar{c}}{\sigma\log(1/3\delta)}
\end{equation}
holds for any $\delta \in (0,1/4)$, where $\bar{c}\equiv\bar{c}(\data)$. Using Lemma~\ref{log.lemma} with the choice
\begin{equation*}
v \coloneqq \left[\min\left\{\log\frac{1}{3\delta},\log\frac{k+d}{u+d}\right\}\right]_{+},
\end{equation*}
we have
\begin{equation}\label{vv}
\mean{B_{2r}}|v-(v)_{B_{2r}}|\,dx \le c(\data).
\end{equation}
Now, since assumption \eqref{density.assumption} is equivalent to
\begin{equation*}
\frac{|B_{2r}\cap\{v=0\}|}{|B_{2r}|} \ge \sigma,
\end{equation*}
a direct modification gives
\begin{equation*}
\log\frac{1}{3\delta} = \frac{1}{|B_{2r}\cap\{v=0\}|}\int_{B_{2r}\cap\{v=0\}}\left(\log\frac{1}{3\delta} - v\right)\,dx \le \frac{1}{\sigma}\left[\log\frac{1}{3\delta}-(v)_{B_{2r}}\right].
\end{equation*}
Integrating this inequality over $B_{2r}\cap\{v=\log(1/3\delta)\}$ and using \eqref{vv}, we find 
\begin{equation*}
\left|B_{2r}\cap\left\{v = \log\frac{1}{3\delta}\right\}\right|\log\frac{1}{3\delta} \le \frac{1}{\sigma}\int_{B_{2r}}|v-(v)_{B_{2r}}|\,dx \le \frac{\bar{c}}{\sigma}|B_{2r}|.
\end{equation*}
Recalling the definition of $v$, we obtain 
\begin{equation*}
\left|B_{2r}\cap \left\{u \le 2\delta k\right\}\right| \le \left|B_{2r}\cap \left\{v = \log\frac{1}{3\delta}\right\}\right| \le \frac{\bar{c}}{\sigma\log(1/3\delta)}|B_{2r}|,
\end{equation*}
which is \eqref{density.est}. We also note that $\delta$ is still free; we will fix its value in the next step. 

\textit{Step 2: A pointwise bound.}
We now prove \eqref{lower.bound}. 
For each $j \in \mathbb{N}\cup\{0\}$, we set
\begin{equation*}
\rho_{j} \coloneqq (1+2^{-j})r, \qquad \tilde{\rho}_{j} \coloneqq \frac{\rho_{j}+\rho_{j+1}}{2}, \qquad B_{j} \coloneqq B_{\rho_{j}}(x_{0})
\end{equation*}
and
\begin{equation*}
\ell_{j} \coloneqq (1+2^{-j})\delta k, \qquad w_{j} \coloneqq (\ell_{j}-u)_{+}, \qquad A_{j} \coloneqq \frac{|B_{j}\cap\{u<\ell_{j}\}|}{|B_{j}|}.
\end{equation*}
Then for any $j$, we have
\begin{equation}\label{parameter.range}
\rho_{j},\tilde{\rho}_{j} \in (r,2r), \qquad d \le \delta k \le \ell_{j} \le 2\delta k, \qquad \ell_{j} - \ell_{j+1} = 2^{-j-1}\delta k \ge 2^{-j-2}\ell_{j}
\end{equation}
and 
\begin{equation}\label{wj}
w_{j} \ge (\ell_{j}-\ell_{j+1})\chi_{\{u<\ell_{j+1}\}} \ge 2^{-j-2}\ell_{j}\chi_{\{u<\ell_{j+1}\}}.
\end{equation}
Moreover, \eqref{density.est} reads as
\begin{equation}\label{A0}
A_{0} = \frac{|B_{0}\cap\{u<\ell_{0}\}|}{|B_{0}|} \le \frac{\bar{c}}{\sigma\log(1/3\delta)}.
\end{equation}

Now, we consider the following two cases separately: $b^+_{B_{2r}} \le \nu/4$ and  $b^+_{B_{2r}} > \nu/4$.

\textbf{Case 1: $b^+_{B_{2r}} \le \nu/4$.}  In this case, we have $a^-_{B_{2r}} \ge 3\nu/4$. Recalling \eqref{G.def}, we simply denote $b^{+}_{B_{2r}} \equiv b^{+}$, $G_{B_{2r}}= G$ and $g_{B_{2r}}=g$. From \eqref{alt.dG}, \eqref{parameter.range}, \eqref{wj} and Lemma~\ref{ineq1} with $a_0=1$ and $b_0=b^{-}_{B_{2r}}$, we have 
\begin{equation}\label{lhs.j+1}
\begin{aligned}
&A_{j+1} G(2^{-j-2}\ell_{j} ) \le A_{j+1} G(\ell_{j}-\ell_{j+1})  =  \frac{1}{|B_{j+1}|}\int_{B_{j+1}\cap\{u < \ell_{j+1}\}}G(\ell_{j}-\ell_{j+1})\,dx \\
& \le  c \mean{B_{j+1}}\left(\frac{w_{j}^p}{r^{sp}}+ b^-\frac{w_{j}^p}{r^{tp}}\right)\,dx  \\
& \le cA_j^{\beta}\left(\mean{B_{j+1}}\int_{B_{j+1}}\left[\frac{|w_{j}(x)-w_{j}(y)|^p}{|x-y|^{n+sp}}+ b^-\frac{|w_{j}(x)-w_{j}(y)|^p}{|x-y|^{n+tp}}\right]\,dx\,dy + \mean{B_{j+1}}G(w_{j})\,dx\right),
\end{aligned}
\end{equation}
where $\beta$ is given in \eqref{def.beta}.
We then estimate the right-hand side of \eqref{lhs.j+1}. It is straightforward to see that
\begin{equation}\label{Gwj}
\mean{B_{j}}G(w_{j})\,dx \le \frac{1}{|B_{j}|}\int_{B_{j}\cap\{u<\ell_{j}\}}G(\ell_{j})\,dx = G(\ell_{j})A_{j}.
\end{equation}
Next, we apply \eqref{caccio.2} to $w_{j}$ and $B_{j}$, which gives 
\begin{equation}\label{Caccio.jth}
\begin{aligned}
&\mean{B_{j+1}}\int_{B_{j+1}}\left(\frac{|w_{j}(x)-w_{j}(y)|^p}{|x-y|^{n+sp}}+ b^-\frac{|w_{j}(x)-w_{j}(y)|^p}{|x-y|^{n+tp}}\right)\,dx\,dy\\
&\le c\mean{B_{j+1}}\int_{B_{j+1}}\left(a(x,y)\frac{|w_{j}(x)-w_{j}(y)|^{p}}{|x-y|^{n+sp}}+b(x,y)\frac{|w_{j}(x)-w_{j}(y)|^{p}}{|x-y|^{n+tp}}\right)\,dx\,dy \\
& \le c\,2^{jp}r^{-p}\mean{B_{j}}\int_{B_{j}}\left(a(x,y)\frac{w_{j}^{p}(x)}{|x-y|^{n+(s-1)p}} + b(x,y)\frac{w_{j}^{p}(x)}{|x-y|^{n+(t-1)p}}\right)\,dx\,dy \\
& \quad\; + c2^{j(n+tp)}T(w_{j};r_{j},r_{j})\mean{B_{j}}w_{j}\,dx.
\end{aligned}
\end{equation}
The first term in the right-hand side of \eqref{Caccio.jth} is estimated as
\begin{equation}\label{rhs.jth}
\begin{aligned}
& 2^{jp}r^{-p}\mean{B_{j}}\int_{B_{j}}\left(a(x,y)\frac{w_{j}^{p}(x)}{|x-y|^{n+(s-1)p}} + b(x,y)\frac{w_{j}^{p}(x)}{|x-y|^{n+(t-1)p}}\right)\,dx\,dy  \\
& \le c\,2^{jp}r^{-p}\ell_{j}^{p}\frac{1}{|B_{j}|}\int_{B_{j}\cap\{u < \ell_{j}\}}\int_{B_{j}}\frac{1}{|x-y|^{n+(s-1)p}}\,dy\,dx \\
& \quad\; + c\,2^{jp}r^{-p}b^{+}\ell_{j}^{p}\frac{1}{|B_{j}|}\int_{B_{j}\cap\{u < \ell_{j}\}}\int_{B_{j}}\frac{1}{|x-y|^{n+(t-1)p}}\,dy\,dx \\
& \le c\,2^{jp}\frac{|B_{j}\cap\{u < \ell_{j}\}|}{|B_{j}|}\left(r^{-sp}\ell_{j}^{p}+ b^{+}r^{-tp}\ell_{j}^{p}\right)\\
& \le c\,2^{jp}G(\ell_{j})A_{j}.
\end{aligned}
\end{equation}
In order to estimate the tail term, we observe that \eqref{a.holder} implies
\begin{equation}\label{a.control}
\begin{aligned}
b(x,y) & \le b(x,y) - b(x,x_{0}) + b^{+} \\
& \le |b(x,y)-b(x,x_{0})|^{(t-s)p/\alpha}(2L)^{1-(t-s)p/\alpha} + b^{+} \\
& \le c|y-x_{0}|^{(t-s)p} + b^{+}
\end{aligned}
\end{equation}
for any $x \in B_{2r}$ and $y \in \mathbb{R}^{n}$. 
Also, using \eqref{T.assumption}, $\eqref{parameter.range}_{2}$ and the fact that $|x-y| \ge |x-x_{0}| - |y-x_{0}| \ge 2^{-1-j}r$ for $x \in \mathbb{R}^{n}\setminus B_{j}$ and $y \in B_{\tilde{\rho}_{j}}$, we estimate the tail term as 
\begin{equation}\label{Tail.jth}
\begin{aligned}
T(w_{j};r_{j},r_{j}) & = \sup_{x \in B_{j}} \int_{\mathbb{R}^{n}\setminus B_{j}}\left(a(x,y)\frac{w_{j}^{p-1}(y) }{|y-x_{0}|^{n+sp}}+b(x,y)\frac{w_{j}^{p-1}(y)}{|y-x_{0}|^{n+tp}}\right)\,dy   \\
& \le \sup_{x \in B_{j}} \int_{\mathbb{R}^{n}\setminus B_{j}}\left(a(x,y)\frac{\ell_{j}^{p-1} }{|y-x_{0}|^{n+sp}}+b(x,y)\frac{\ell_{j}^{p-1}}{|y-x_{0}|^{n+tp}}\right)\,dy   \\
& \quad\; + \sup_{x \in B_{j}} \int_{\mathbb{R}^{n}\setminus B_{R}}\left(a(x,y)\frac{(u_{-}(y))^{p-1} }{|y-x_{0}|^{n+sp}}+b(x,y)\frac{(u_{-}(y))^{p-1}}{|y-x_{0}|^{n+tp}}\right)\,dy   \\
& \le \ell_{j}^{p-1}\int_{\mathbb{R}^{n}\setminus B_{j}}\left(\frac{1 }{|y-x_{0}|^{n+sp}}+b^+\frac{1}{|y-x_{0}|^{n+tp}}\right)\,dy   + T(u_-;R,2r)   \\
& \le c \left(g(\ell_{j})+g(d)\right) \\
& \le c g(\ell_{j}).
\end{aligned}
\end{equation}
Also, by the definitions of $w_j$ and $A_{j}$, we directly have
\begin{equation}\label{wjmean}
\mean{B_{j}}w_{j}\,dx \le \ell_{j}A_{j}.
\end{equation}
Combining all the estimates \eqref{lhs.j+1}, \eqref{Gwj}, \eqref{Caccio.jth}, \eqref{rhs.jth}, \eqref{Tail.jth} and \eqref{wjmean}, we obtain
\begin{equation*}
A_{j+1} \le c_{0}2^{j(n+tp+2p)}A_{j}^{1+\beta}
\end{equation*}
for a constant $c_{0}\equiv c_{0}(\data)>0$.
By choosing
\begin{equation*}
\delta = \frac{1}{4}\exp\left(-\frac{\bar{c}}{\sigma}c_{0}^{1/\beta}2^{(n+tp+2p)\beta^{2}}\right) \in \left(0,\frac{1}{4}\right)
\end{equation*} 
in \eqref{A0}, we have
\begin{equation*}
A_{0} \le \frac{\bar{c}}{\sigma\log(1/3\delta)} \le c_{0}^{-1/\beta}2^{-(n+tp+2p)/\beta^{2}}.
\end{equation*} 
Therefore we can apply Lemma~\ref{iterlem} to conclude that $A_{j} \rightarrow 0$, from which \eqref{lower.bound} follows.

\textbf{Case 2: $b^+_{B_{2r}}>\nu/4$.}  In this case, observe that $b^-_{B_{2r}} \ge   b^+_{B_{2r}} - [b]_{\alpha}(2R_0)^{\alpha} \ge \nu/8$ from \eqref{R0.choice}.
From \eqref{parameter.range}, \eqref{wj} and Lemma~\ref{ineq1} with $a_0=0$ and $b_0=1$, we have 
\begin{equation}\label{lhs.j+11}
\begin{aligned}
A_{j+1} (2^{-j-2}\ell_{j} )^p &\le A_{j+1} (\ell_{j}-\ell_{j+1})^p  =  \frac{1}{|B_{j+1}|}\int_{B_{j+1}\cap\{u < \ell_{j}\}}(\ell_{j}-\ell_{j+1})^p\,dx \\
& \le cr^{tp}A_j^{\beta}\left(\mean{B_{j+1}}\int_{B_{j+1}}\frac{|w_{j}(x)-w_{j}(y)|^p}{|x-y|^{n+tp}}\,dx\,dy + \mean{B_{j}}w_{j}^p\,dx\right).
\end{aligned}
\end{equation}
We then estimate each term in the right-hand side. It is straightforward to see that
\begin{equation}\label{Gwj1}
\mean{B_{j}}w_{j}^p\,dx \le \frac{1}{|B_{j}|}\int_{B_{j}\cap\{u<\ell_{j}\}}\ell_{j}^p\,dx \le \ell_{j}^pA_{j}.
\end{equation}
Next, we apply \eqref{caccio.2} to $w_{j}$ and $B_{j}$, which gives 
\begin{equation}\label{Caccio.jth1}
\begin{aligned}
&\mean{B_{j+1}}\int_{B_{j+1}}\frac{|w_{j}(x)-w_{j}(y)|^p}{|x-y|^{n+tp}}\,dx\,dy\\
&\le \mean{B_{j+1}}\int_{B_{j+1}}\left(a(x,y)\frac{|w_{j}(x)-w_{j}(y)|^{p}}{|x-y|^{n+sp}}+b(x,y)\frac{|w_{j}(x)-w_{j}(y)|^{p}}{|x-y|^{n+tp}}\right)\,dx\,dy \\
& \le c\,2^{jp}r^{-p}\mean{B_{j}}\int_{B_{j}}\left(a(x,y)\frac{w_{j}^{p}(x)}{|x-y|^{n+(s-1)p}} + b(x,y)\frac{w_{j}^{p}(x)}{|x-y|^{n+(t-1)p}}\right)\,dx\,dy \\
& \quad\; + c2^{j(n+tp)}T(w_{j};r_{j},r_{j})\mean{B_{j}}w_{j}\,dx.
\end{aligned}
\end{equation}
Similarly to \eqref{rhs.jth} and \eqref{Tail.jth}, we can estimate each term in the right-hand side as
\begin{equation}\label{rhs.jth1}
\begin{aligned}
& 2^{jp}r^{-p}\mean{B_{j}}\int_{B_{j}}\left(a(x,y)\frac{w_{j}^{p}(x)}{|x-y|^{n+(s-1)p}} + b(x,y)\frac{w_{j}^{p}(x)}{|x-y|^{n+(t-1)p}}\right)\,dx\,dy  \\
& \le c2^{jp}\frac{|B_{j}\cap\{u < \ell_{j}\}|}{|B_{j}|}\left(r^{-sp}\ell_{j}^{p}+ r^{-tp}\ell_{j}^{p}\right)\\
& \le c 2^{jp}r^{-tp}\ell_{j}^pA_{j},
\end{aligned}
\end{equation}
and
\begin{equation}\label{Tail.jth1}
T(w_{j};r_{j},r_{j}) \le c \left(r^{-sp}+r^{-tp}\right)\ell_{j}^{p-1}+ c g(d)  \le c r^{-tp}\ell_{j}^{p-1}.
\end{equation}
Combining all the estimates \eqref{lhs.j+11}, \eqref{Gwj1}, \eqref{Caccio.jth1}, \eqref{rhs.jth1}, \eqref{Tail.jth1} and \eqref{wjmean}, we obtain
\begin{equation*}
A_{j+1} \le c_{1}2^{j(n+tp+2p)}A_{j}^{1+\beta}
\end{equation*}
for a constant $c_{1}\equiv c_{1}(\data)>0$.
Therefore, as in \textbf{Case 1}, we can choose $\delta$ sufficiently small in order to have $A_{j} \rightarrow 0$, thereby concluding with \eqref{lower.bound}.
\end{proof}

\section{Proof of Theorems \ref{thm.hol} and \ref{thm.harnack}}\label{sec.mainthm}

In this section, we keep on assuming \eqref{a.holder} and \eqref{assumption.hol} on $b(\cdot,\cdot)$. 

\subsection{H\"older regularity}
Here we prove H\"older regularity for \eqref{main.eq}.

\begin{proof}[Proof of Theorem~\ref{thm.hol}]
Let $B_{r}\equiv B_{r}(x_{0})\subset \Omega$ be a fixed ball, and set
\[ k_{0} \coloneqq 2\|u\|_{L^{\infty}(B_{r/2})} + 2\left[\frac{1}{G_{B_{r}}(1)}T(u;r/2,r)\right]^{1/(p-1)} = 2\|u\|_{L^{\infty}(B_{r/2})} + 2g_{B_{r}}^{-1}(T(u;r/2,r)). \]
We claim that there exist small constants $\gamma \in (0,1)$ and $\tau \in (0,1)$, both depending only on $\data$, such that
\begin{equation}\label{osc.est}
\essosc_{B_{\tau^{j}r/2}}u \le \tau^{\gamma j}k_{0}
\end{equation}
holds for every $j \in \mathbb{N}$. 
We show this by using strong induction on $j$. First, the definition of $k_{0}$ directly implies that \eqref{osc.est} holds for $j=0$. Now, with $i \in \mathbb{N}\cup\{0\}$ being fixed, we assume that \eqref{osc.est} holds for all $j\in\{0,1,\ldots,i\}$, and then show that it holds for $j=i+1$.

For each $j \in \mathbb{N}\cup\{0\}$, we set
\[ r_{j} \coloneqq \tau^{j}\frac{r}{2}, \qquad B_{j} \coloneqq B_{r_{j}}(x_{0}), \qquad k_{j} \coloneqq \left(\frac{r_{j}}{r_{0}}\right)^{\gamma}k_{0} = \tau^{\gamma j}k_{0}, \]
where $\gamma \in (0,1)$ and $\tau \in (0,1)$ are free parameters whose values will be chosen later, and then
\[ M_{j} \coloneqq \esssup_{B_{j}}u, \qquad m_{j} \coloneqq \essinf_{B_{j}}u. \]
Observe that either
\begin{equation}\label{alt}
\left|2B_{i+1}\cap\left\{u-m_{i} \ge \frac{k_{i}}{2}\right\}\right| \ge \frac{1}{2}\left|2B_{i+1}\right|
\quad \text{or} \quad
\left|2B_{i+1}\cap\left\{u-m_{i} < \frac{k_{i}}{2}\right\}\right| \ge \frac{1}{2}\left|2B_{i+1}\right|
\end{equation}
must hold; we accordingly set
\begin{equation*}
u_{i} \coloneqq 
\begin{cases}
u - m_{i} & \text{ if $\eqref{alt}_{1}$ holds,} \\
k_{i} - (u-m_{i}) & \text{ if $\eqref{alt}_{2}$ holds}.
\end{cases}
\end{equation*}
In any case, $u_{i}$ is a weak solution to \eqref{main.eq}, $u_{i} \ge 0$ in $B_{i}$ and
\begin{equation}\label{levelset.i+1}
\left|2B_{i+1}\cap\left\{u_{i} \ge \frac{k_{i}}{2}\right\}\right| \ge \frac{1}{2}\left|2B_{i+1}\right|.
\end{equation}
It moreover satisfies
\[ |u_{i}| \le M_{j} - m_{j} + k_{i} \le k_{j} + k_{i} \le 2k_{j} \;\; \text{a.e. in}\;\; B_{j} \quad \text{for any}\;\; j \in \{0,1,\ldots,i\}, \]
and 
\[ |u_{i}| \le |u| + 2k_{0} \quad  \text{a.e. in}\;\; \mathbb{R}^{n}\setminus B_{0}. \] 
From these observations, we estimate
\begin{equation*}
\begin{aligned}
& T(u_{i};r_{i},2r_{i+1}) \\
& \le \sup_{x\in 2B_{i+1}}\sum_{j=1}^{i}\int_{B_{j-1}\setminus B_{j}}\left(a(x,y)\frac{|u_{i}(y)|^{p-1}}{|y-x_{0}|^{n+sp}}+b(x,y)\frac{|u_{i}(y)|^{p-1}}{|y-x_{0}|^{n+tp}}\right)\,dy \\
& \qquad  + \sup_{x\in 2B_{i+1}}\int_{\mathbb{R}^{n}\setminus B_{0}}\left(a(x,y)\frac{|u_{i}(y)|^{p-1}}{|y-x_{0}|^{n+sp}}+b(x,y)\frac{|u_{i}(y)|^{p-1}}{|y-x_{0}|^{n+tp}}\right)\,dy \\
& \le \sup_{x\in 2B_{i+1}}\sum_{j=1}^{i}\int_{B_{j-1}\setminus B_{j}}\left(a(x,y)\frac{(2k_{j-1})^{p-1}}{|y-x_{0}|^{n+sp}}+b(x,y)\frac{(2k_{i-1})^{p-1}}{|y-x_{0}|^{n+tp}}\right)\,dy \\
& \qquad + \sup_{x\in 2B_{i+1}}\int_{\mathbb{R}^{n}\setminus B_{0}}\left(a(x,y)\frac{(|u(y)|+2k_{0})^{p-1}}{|y-x_{0}|^{n+sp}}+b(x,y)\frac{(|u(y)|+2k_{0})^{p-1}}{|y-x_{0}|^{n+tp}}\right)\,dy.
\end{aligned}
\end{equation*}
To proceed, we again divide the cases.

\textbf{Case 1: $b^{+}_{2B_{i+1}}\le \nu/4$. } In this case, we use \eqref{a.bound}, \eqref{a.control} and the fact that
\[ 2r_{j+1} \le 2r_{0}, \ b^{+}_{2B_{0}} \le b^{+}_{2B_{i+1}}+ [b]_{\alpha}(2r_{0})^{\alpha} \le b^{+}_{2B_{i+1}} + [b]_{\alpha}(2r_{0})^{(t-s)p} \;\; \Longrightarrow \;\; G_{2B_{0}}(1) \lesssim G_{2B_{i+1}}(1) \]
in order to get
\begin{equation*}
\begin{aligned}
& T(u_{i};r_{i},2r_{i+1}) \\
& \le c\sum_{j=1}^{i}(2k_{j-1})^{p-1}\left(r_{j}^{-sp} + b^{+}_{2B_{i+1}}r_{j}^{-tp}\right)  + ck_{0}^{p-1}G_{2B_{0}}(1) + ck_{0}^{p-1}\left(r_{0}^{-sp}+b^{+}_{2B_{i+1}}r_{0}^{-tp}\right) \\
& \le c\sum_{j=1}^{i}k_{j-1}^{p-1}\left(r_{j}^{-sp}+b^{+}_{2B_{i+1}}r_{j}^{-tp}\right).
\end{aligned}
\end{equation*} 
In turn, we have
\begin{equation*}
\begin{aligned}
\frac{1}{G_{2B_{i+1}}(1)}T(u_{i};r_{i},2r_{i+1}) & \le c\sum_{j=1}^{i}\frac{r_{j}^{-sp}+b^{+}_{2B_{i+1}}r_{j}^{-tp}}{r_{i+1}^{-sp}+b^{+}_{2B_{i+1}}r_{i+1}^{-tp}}k_{j-1}^{p-1} \\
& \le c\sum_{j=1}^{i}\left(\frac{r_{i+1}}{r_{j}}\right)^{sp}k_{j-1}^{p-1}  = ck_{i}^{p-1}\sum_{j=1}^{i}\tau^{(i+1-j)[sp-\gamma(p-1)]}.
\end{aligned}
\end{equation*}

\textbf{Case 2: $b^{+}_{2B_{i+1}} > \nu/4$. } In this case, with \eqref{a.bound}, we directly estimate
\begin{equation*}
\begin{aligned}
 T(u_{i};r_{i},2r_{i+1}) 
& \le c\sum_{j=1}^{i}(2k_{j-1})^{p-1}\left(r_{j}^{-sp} + r_{j}^{-tp}\right)  + ck_{0}^{p-1}G_{2B_{0}}(1) + ck_{0}^{p-1}\left(r_{0}^{-sp}+ r_{0}^{-tp}\right) \\
& \le c\sum_{j=1}^{i}k_{j-1}^{p-1}\left(r_{j}^{-sp}+r_{j}^{-tp}\right)  \le c\sum_{j=1}^{i}k_{j-1}^{p-1}r_{j}^{-tp}
\end{aligned}
\end{equation*}
and so
\begin{equation*}
\frac{1}{G_{2B_{i+1}}(1)}T(u_{i};r_{i},2r_{i+1}) 
 \le c\sum_{j=1}^{i}\left(\frac{r_{i+1}}{r_{j}}\right)^{tp}k_{j-1}^{p-1}  = ck_{i}^{p-1}\sum_{j=1}^{i}\tau^{(i+1-j)[tp-\gamma(p-1)]}.
\end{equation*}

Therefore, if $\gamma \in (0,1)$ is so small that
\begin{equation}\label{gamma.choice1}
\gamma \le \frac{sp}{2(p-1)} \le \frac{tp}{2(p-1)},
\end{equation}
then in any case we have
\begin{equation*}
\frac{1}{G_{2B_{i+1}}(1)}T((u_{i})_{-};r_{i},2r_{i+1}) \le \frac{1}{G_{2B_{i+1}}(1)}T(u_{i};r_{i},2r_{i+1}) \le ck_{i}^{p-1}\sum_{j=1}^{i}\tau^{jsp/2} \le c\frac{\tau^{sp/2}}{1-\tau^{sp/2}}k_{i}^{p-1}
\end{equation*}
for a constant $c\equiv c(\data)>0$. We now choose $\tau \equiv \tau(\data) \in (0,1/8)$ so small that
\begin{equation*} 
\left[\frac{1}{G_{2B_{i+1}}(1)}T((u_{i})_{-};r_{i},2r_{i+1})\right]^{1/(p-1)} \le \left(c\frac{\tau^{sp/2}}{1-\tau^{sp/2}}\right)^{1/(p-1)}k_{i} \le \frac{\delta}{2}k_{i},
\end{equation*}
where $\delta \equiv \delta(\data) \in (0,1/4)$ is the constant determined in Lemma~\ref{density.improve} with the choice $\sigma \equiv 1/2$. 
Consequently, recalling \eqref{levelset.i+1}, we employ Lemma~\ref{density.improve} with the choices $k \equiv k_{i}/2$, $\sigma\equiv 1/2$, $B_{R} \equiv B_{i}$ and $B_{r} \equiv B_{i+1}$, which gives
\[ \essinf_{B_{i+1}}u_{i+1} \ge \frac{\delta k_{i}}{2}. \] 
In turn, we have the following:
\begin{itemize}
\item[(i)] If $\eqref{alt}_{1}$ holds, then $m_{i+1} - m_{i} \ge \delta k_{i}/2$ and so 
\begin{equation*}
M_{i+1} - m_{i+1} \le M_{i} - m_{i} - (m_{i+1}-m_{i}) = \essosc_{B_{i}}u - (m_{i+1}-m_{i}) \le \left(1-\frac{\delta}{2}\right)k_{i}.
\end{equation*}
\item[(ii)] If $\eqref{alt}_{2}$ holds, then $k_{i} - M_{i+1} + m_{i} \ge \delta k_{i}/2$ and so
\begin{equation*}
M_{i+1}-m_{i+1} \le M_{i+1}-m_{i} \le \left(1-\frac{\delta}{2}\right)k_{i}.
\end{equation*}
\end{itemize}
Namely, in any case we obtain 
\[ \essosc_{B_{i+1}}u \le \left(1-\frac{\delta}{2}\right)\tau^{-\gamma}k_{i+1}. \]
Then we finally fix $\gamma\in (0,1)$ as a small constant, depending only on $\data$, in a way that \eqref{gamma.choice1} and $1-\delta/2\le \tau^{\gamma}$ are satisfied; this leads to \eqref{osc.est} for $j=i+1$. Hence, we conclude with
\begin{equation*}
\begin{aligned}
\essosc_{B_{j}}u & \le c\tau^{\gamma j}\left[ \|u\|_{L^{\infty}(B_{r/2})} + g_{B_{r}}^{-1}(T(u;r/2,r)) \right] \\
\overset{\eqref{est:bdd2}}&{\le} c\tau^{\gamma j}\left[\left(\mean{B_{r}}|u|^{p}\,dx\right)^{1/p} + g_{B_{r}}^{-1}(T(u;r/2,r))\right]
\end{aligned}
\end{equation*}
for any $j\in\mathbb{N}$, and then estimate \eqref{est:hol} follows by elementary manipulations.
\end{proof}

\subsection{Harnack inequality}
In the following, we further assume that $a(\cdot,\cdot)\equiv1$. 
We first recall a Krylov-Safonov type covering lemma, see \cite{KS}. The following version, which employs balls instead of cubes, can be found in \cite[Lemma~7.2]{KS01}. 
\begin{lemma}\label{KrySaf}
Let $E \subset B_{R}$ be a measurable set, and let $\bar{\delta} \in (0,1)$. Define
\begin{equation}\label{Edelta}
[E]_{\bar{\delta}} \coloneqq \bigcup \left\{ B_{3\rho}(x) \cap B_{R} : x\in B_{R}, \rho>0, |E\cap B_{3\rho}(x)| > \bar{\delta}|B_{\rho}(x)| \right\}.
\end{equation}
Then either 
$[E]_{\bar{\delta}} = B_{R}$, or else $|[E]_{\bar{\delta}}| \ge (2^{n}\bar{\delta})^{-1}|E|$.
\end{lemma}

\begin{remark}\label{rmk.rho}
It is clear that, in the above definition of $[E]_{\bar{\delta}}$, we may consider only the balls $B_{3\rho}(x)$ with $\rho \le 2r/3$.
\end{remark}

Using Lemmas~\ref{density.improve} and \ref{KrySaf}, we obtain a weak Harnack type estimate.
\begin{lemma}\label{weak.harnack}
Let $u\in\mathcal{A}(\Omega) \cap \mathcal{T}(\mathbb{R}^{n})$ be a weak supersolution to \eqref{main.eq} under assumptions \eqref{kernel.growth}--\eqref{assumption.hol} with $a(\cdot,\cdot)\equiv 1$, which is nonnegative in a ball $B_{10r} \equiv B_{10r}(x_{0}) \Subset \Omega$ with $10r \le R_{0}$. Then there exist constants $p_{0} \in (0,1)$ and $c \ge 1$, both depending only on $\data$, such that 
\begin{equation*}
\left(\mean{B_{2r}}u^{p_{0}}\,dx\right)^{1/p_{0}} \le c \essinf_{B_{r}}u + cg_{B_{6r}}^{-1}(T(u_{-};10r,10r)). 
\end{equation*}
\end{lemma}
\begin{proof}
We set $\bar{\delta} \coloneqq 2^{-n-1}$ and $T_{0}\coloneqq c_{2}g_{B_{6r}}^{-1}(T(u_{-};10r,10r))$, with $c_{2}\equiv c_{2}(\data)$ being a constant whose precise value will be determined in \eqref{dT} below. 
For each $i \in \mathbb{N}\cup\{0\}$ and $\tau>0$, we define
\begin{equation*}
A^{i}_{\tau} \coloneqq \left\{ x\in B_{2r}:u(x)>\tau\delta^{i}-\frac{T_{0}}{1-\delta}\right\},
\end{equation*}
where $\delta \equiv \delta(\data) \in (0,1/4)$ is the constant determined in Lemma~\ref{density.improve} in the case $\sigma=\bar{\delta}/6^{n}$. 
It is obvious that $A^{i-1}_{\tau} \subset A^{i}_{\tau}$.

Let $x \in B_{2r}$ and $\rho \in (0,2r/3]$ be such that
\[ B_{3\rho}(x)\cap B_{2r} \subset [A^{i-1}_{\tau}]_{\bar{\delta}} \]
holds (recall the notation \eqref{Edelta} and Remark~\ref{rmk.rho}). Then we get
\begin{equation*}
|A^{i-1}_{\tau}\cap B_{3\rho}(x)| \ge \bar{\delta}|B_{\rho}(x)|
\end{equation*}
and so
\begin{equation*}
\frac{|A^{i-1}_{\tau} \cap B_{6\rho}(x)|}{|B_{6\rho}(x)|} \ge \frac{1}{6^{n}}\frac{|A^{i-1}_{\tau}\cap B_{3\rho}(x)|}{|B_{\rho}(x)|} \ge \frac{\bar{\delta}}{6^{n}}.
\end{equation*}
Moreover, $u \ge 0$ in $B_{12\rho}(x)\subset B_{10r}$. 
Therefore we can apply Lemma~\ref{density.improve} to $u$ on the concentric balls $B_{3\rho}(x) \subset B_{6\rho}(x) \subset B_{12\rho}(x)$ with the choices $k \equiv \tau\delta^{i-1} - T_{0}/(1-\delta)$, $\sigma \equiv \bar{\delta}/6^{n}$ and $d \equiv g_{B_{6\rho}(x)}^{-1}(T(u_{-};x,12\rho,6\rho))$; this leads to
\begin{equation*}
\essinf_{B_{3\rho}(x)}u \ge \delta\left(\tau\delta^{i-1}-\frac{T_{0}}{1-\delta}\right) - d.
\end{equation*}
We also observe that, since $u \ge 0$ in $B_{10r}$, 
\begin{equation}\label{dT}
\begin{aligned}
g_{B_{6\rho}(x)}(d) & = \sup_{B_{6\rho}(x)}\int_{\mathbb{R}^{n}\setminus B_{12\rho}(x)}\left(\frac{(u_{-}(y))^{p-1}}{|y-x|^{n-sp}}+b(\cdot,y)\frac{(u_{-}(y))^{p-1}}{|y-x|^{n+tp}}\right)\,dy \\
& \le c\sup_{B_{10r}}\int_{\mathbb{R}^{n}\setminus B_{10r}}\left(\frac{(u_{-}(y))^{p-1}}{|y-x_{0}|^{n+sp}}+b(\cdot,y)\frac{(u_{-}(y))^{p-1}}{|y-x_{0}|^{n+tp}}\right)\,dy \\
& = cT(u_{-};10r,10r) \\
& \le c_{2}^{p-1}g_{B_{6\rho}(x)}(g_{B_{6r}}^{-1}(T(u_{-};10r,10r))) = g_{B_{6\rho}(x)}(T_{0})
\end{aligned}
\end{equation}
holds for a constant $c_{2}\equiv c_{2}(\data)>0$, where we have also used the fact that
\[ \rho \le 2r/3, \ b^{+}_{B_{6r}} \lesssim b^{+}_{B_{6\rho}(x)} + r^{\alpha} \lesssim b^{+}_{B_{6\rho}(x)} + r^{(t-s)p}  \;\; \Longrightarrow \;\; g_{B_{6r}}(\cdot) \lesssim g_{B_{6\rho}(x)}(\cdot). \] 
Consequently, we have
\[ \essinf_{B_{3\rho}(x)}u \ge \delta\left(\tau\delta^{i-1}-\frac{T_{0}}{1-\delta}\right)-T_{0} = \tau\delta^{i} - \frac{T_{0}}{1-\delta} \] 
and this implies $B_{3\rho}(x)\cap B_{2r} \subset A^{i}_{\tau}$. In turn, since $B_{3\rho}(x)$ was an arbitrary member of the family making the union in \eqref{Edelta}, we have $[A^{i-1}_{\tau}]_{\bar{\delta}} \subset A^{i}_{\tau}$. 
We now apply Lemma~\ref{KrySaf} with $E=A^{i-1}_{\tau}$, and the rest of the proof is exactly the same as the case of fractional $p$-Laplacian \cite{coz,DKP14}. We thus finish the proof here.
\end{proof}

The following lemma allows us to capture the precise tail contribution in the  Harnack inequality.
\begin{lemma}\label{tail.pm}
Let $u \in \mathcal{A}(\Omega) \cap \mathcal{T}(\mathbb{R}^{n})$ be a weak solution to \eqref{main.eq} under assumptions \eqref{kernel.growth}--\eqref{assumption.hol} with $a(\cdot,\cdot) \equiv 1$, which is nonnegative in a ball $B_{R}\equiv B_{R}(x_{0}) \Subset \Omega$ with $R\le R_{0}$. Then 
\[ T(u_{+};r,2r) \le cg_{B_{r}}\left(\sup_{B_{r}}u\right) + cT(u_{-};R,r) \]
holds for any $r\in(0,R/2]$, where $c\equiv c(\data)>0$.
\end{lemma}
\begin{proof}
We set $w \coloneqq u-2k$ for $k \coloneqq \sup_{B_{r}}u$, and take a cut-off function $\phi \in C^{\infty}_{0}(B_{r})$ such that $\supp\,\phi \subseteq B_{3r/4}$, $0 \le \phi \le 1$, $\phi \equiv 1$ in $B_{r/2}$ and $|D\phi| \le 8/r$. Recalling the notation \eqref{phip}, we test \eqref{main.eq} with $w\phi^{p}$ in order to have
\begin{equation}\label{test}
\begin{aligned}
0 & = \int_{B_{r}}\int_{B_{r}}\Phi_{p}(u(x)-u(y))(w(x)\phi^{p}(x)-w(y)\phi^{p}(y))\left[K_{sp}(x,y) + b(x,y)K_{tp}(x,y)\right] \,dx\,dy  \\
& \quad\; + 2\int_{\mathbb{R}^{n}\setminus B_{r}}\int_{B_{r}}\Phi_{p}(u(x)-u(y))w(x)\phi^{p}(x)\left[K_{sp}(x,y) + b(x,y)K_{tp}(x,y)\right]\,dx\,dy  \\
& \eqqcolon I_{1} + I_{2}.
\end{aligned}
\end{equation}
We first estimate $I_{1}$. By the calculations done in the proof of \cite[Lemma 4.2]{BOS}, we have
\begin{equation*}
\begin{aligned}
& \Phi_{p}(w(x)-w(y))(w(x)\phi^{p}(x)-w(y)\phi^{p}(y)) \\
& \ge \frac{1}{4}|w(x)-w(y)|^{p}(\phi^{p}(x)+\phi^{p}(y)) - c|\phi(x)-\phi(y)|^{p}(w(x)+w(y))^{p} \\
& \ge -ck^{p}\frac{|x-y|^{p}}{r^{p}}
\end{aligned}
\end{equation*}
and therefore
\begin{equation}\label{I1}
\begin{aligned}
I_{1} & \ge -c\int_{B_{r}}\int_{B_{r}}\left(\frac{k^{p}}{r^{p}}\frac{1}{|x-y|^{n+(s-1)p}} + b(x,y)\frac{k^{p}}{r^{p}}\frac{1}{|x-y|^{n+(t-1)p}}\right)\,dx\,dy \\
& \ge -c\left(k^{p}r^{n-sp} + b^{+}_{B_{r}}k^{p}r^{n-tp}\right) \\
& \ge -cr^{n}G_{B_{r}}(k).
\end{aligned}
\end{equation}
We next split $I_{2}$ as
\begin{equation}\label{I2.split}
\begin{aligned}
I_{2} & \ge 2 \int_{\mathbb{R}^{n}\setminus B_{r}}\int_{B_{r}}k(u(y)-k)_{+}^{p-1}\phi^{p}(x)\left[K_{sp}(x,y) + b(x,y)K_{tp}(x,y)\right]\,dx\,dy \\
& \quad\; -2\int_{\mathbb{R}^{n}\setminus B_{r}}\int_{B_{r}}2k\chi_{\{u(y)<k\}}(u(x)-u(y))_{+}^{p-1}\phi^{p}(x)\left[K_{sp}(x,y) + b(x,y)K_{tp}(x,y)\right]\,dx\,dy \\
& \eqqcolon I_{2,1} - I_{2,2}.
\end{aligned}
\end{equation}
We note that 
\[ \frac{1}{4}|y-x_{0}| \le |y-x_{0}| - |x-x_{0}| \le |x-y| \le |x-x_{0}| + |y-x_{0}| \le \frac{7}{4}|y-x_{0}|\] 
for any $x\in B_{3r/4}$ and $y\in\mathbb{R}^{n}\setminus B_{r}$,
in order to estimate $I_{2,1}$ as
\begin{equation}\label{I21}
\begin{aligned}
I_{2,1} & \ge \frac{k}{c}\int_{\mathbb{R}^{n}\setminus B_{r}}\int_{B_{r}}\left(\frac{(u_{+}(y))^{p-1}}{|x-y|^{n+sp}}+b(x,y)\frac{(u_{+}(y))^{p-1}}{|x-y|^{n+tp}}\right)\phi^{p}(x)\,dx\,dy \\
& \quad\; - c\int_{\mathbb{R}^{n}\setminus B_{r}}\int_{B_{r}}\left(\frac{k^{p}}{|x-y|^{n+sp}} + b(x,y)\frac{k^{p}}{|x-y|^{n+tp}}\right)\phi^{p}(x)\,dx\,dy \\
& \ge \frac{kr^{n}}{c}\int_{\mathbb{R}^{n}\setminus B_{r}}\left(\frac{(u_{+}(y))^{p-1}}{|y-x_{0}|^{n+sp}} + \inf_{B_{r/2}}b(\cdot,y)\frac{(u_{+}(y))^{p-1}}{|y-x_{0}|^{n+tp}}\right)\,dy - cr^{n}G_{B_{r}}(k).
\end{aligned}
\end{equation}
On the other hand, we observe that
\begin{itemize}
\item[(i)] If $x\in B_{r}$ and $y \in B_{R}$, then $(u(x)-u(y))_{+}^{p-1} \le (u(x))^{p-1}$;

\item[(ii)] If $x\in B_{r}$ and $y \in \mathbb{R}^{n} \setminus B_{R}$, then $(u(x)-u(y))_{+}^{p-1} \le 2^{p-1}\left[ (u(x))^{p-1} + (u_{-}(y))^{p-1} \right]$.
\end{itemize}
We thus estimate $I_{2,2}$ as
\begin{equation}\label{I22}
\begin{aligned}
I_{2,2} & \le 2k\int_{B_{R}\setminus B_{r}}\int_{B_{r}}\left(\frac{k^{p-1}}{|y-x_{0}|^{n+sp}} + b(x,y)\frac{k^{p-1}}{|y-x_{0}|^{n+tp}}\right)\phi^{p}(x)\,dx\,dy \\
& \quad\; + 2k\int_{\mathbb{R}^{n}\setminus B_{R}}\int_{B_{r}}\left(\frac{(k+u_{-}(y))^{p-1}}{|y-x_{0}|^{n+sp}} + b(x,y)\frac{(k+u_{-}(y))^{p-1}}{|y-x_{0}|^{n+tp}}\right)\phi^{p}(x)\,dx\,dy \\
& \le  cr^{n}G_{B_{r}}(k) + ckr^{n}T(u_{-};R,r).
\end{aligned}
\end{equation}
Connecting \eqref{I1}, \eqref{I2.split}, \eqref{I21} and \eqref{I22} to \eqref{test},  we arrive at
\begin{equation*}
\int_{\mathbb{R}^{n}\setminus B_{r}}\left(\frac{(u_{+}(y))^{p-1}}{|y-x_{0}|^{n+sp}} +\inf_{B_{r/2}}b(\cdot,y)\frac{(u_{+}(y))^{p-1}}{|y-x_{0}|^{n+tq}}\right)\,dy  \le cg_{B_{r}}\left(\sup_{B_{r}}u\right) + cT(u_{-};R,r).
\end{equation*}
Finally, we estimate
\begin{equation*}
\begin{aligned}
& T(u_{+};r,2r) = \sup_{x \in B_{2r}}\int_{\mathbb{R}^{n}\setminus B_{r}}\left(\frac{(u_{+}(y))^{p-1}}{|y-x_{0}|^{n+sp}} + b(x,y)\frac{(u_{+}(y))^{p-1}}{|y-x_{0}|^{n+tp}}\right)\,dy\\
& \le \int_{\mathbb{R}^{n}\setminus B_{r}}\left(\frac{(u_{+}(y))^{p-1}}{|y-x_{0}|^{n+sp}} + \sup_{B_{2r}} b(\cdot,y)\frac{(u_{+}(y))^{p-1}}{|y-x_{0}|^{n+tp}}\right)\,dy\\
& \le \int_{\mathbb{R}^{n}\setminus B_{r}}\left(\frac{(u_{+}(y))^{p-1}}{|y-x_{0}|^{n+sp}} + \inf_{B_{2r}} b(\cdot,y)\frac{(u_{+}(y))^{p-1}}{|y-x_{0}|^{n+tp}}\right)\,dy
 + cr^{\alpha} \int_{\mathbb{R}^{n}\setminus B_{r/2}} \frac{(u_{+}(y))^{p-1}}{|y-x_{0}|^{n+sp}r^{tp-sp}}\,dy\\
& \le c \int_{\mathbb{R}^{n}\setminus B_{r}}\left(\frac{(u_{+}(y))^{p-1}}{|y-x_{0}|^{n+sp}} + \inf_{B_{r/2}} b(\cdot,y)\frac{(u_{+}(y))^{p-1}}{|y-x_{0}|^{n+tp}}\right)\,dy.
\end{aligned}
\end{equation*}
Combining the last two displays, we conclude with the desired estimate.
\end{proof}

We are now in a position to prove Theorem~\ref{thm.harnack}. 

\begin{proof}[Proof of Theorem~\ref{thm.harnack}]
Let $u$ be a weak solution to \eqref{main.eq} which is nonnegative in a ball $B_{10r} \equiv B_{10r}(x_{0})\Subset\Omega$ with $10r \le R_{0}$.  
First, arguing in a way completely similar to the proof of \eqref{est:bdd2}, we have
\begin{equation*}
\sup_{B_{\sigma r}}u \le \frac{c_{\varepsilon}}{(\tau-\sigma)^{n/p}}\left(\mean{B_{\tau r}}u^{p}\,dx\right)^{1/p} + \varepsilon g_{B_{r}}^{-1}\left( T(u_{+};r,2r)\right)
\end{equation*}
whenever $1 \le \sigma < \tau \le 2$ and $\varepsilon \in (0,1]$, where $c_{\varepsilon}\equiv c_{\varepsilon}(\data,\varepsilon)>0$. 
Next, using Young's inequality, we estimate
\begin{equation*}
\begin{aligned}
\sup_{B_{\sigma r}} u
& \le \frac{c_{\varepsilon}}{(\tau-\sigma)^{n/p}}\left(\sup_{B_{\tau r}}u\right)^{(p-p_{0})/p}\left(\mean{B_{2r}}u^{p_{0}}\,dx\right)^{1/p} + \varepsilon g_{B_{r}}^{-1}(T(u_{+};r,2r)) \\
&\le \frac{1}{2}\sup_{B_{\tau r}}u + \frac{c_{\varepsilon}}{(\tau-\sigma)^{n/p_{0}}}\left(\mean{B_{2r}}u^{p_{0}}\,dx\right)^{1/p_{0}}+ \varepsilon g_{B_{r}}^{-1}(T(u_{+};r,2r)) 
\end{aligned}
\end{equation*}
whenever $1 \le \sigma < \tau \le 2$, where $p_{0}\equiv p_{0}(\data)$ is the constant determined in Lemma~\ref{weak.harnack}. Then Lemma~\ref{techlem} implies
\begin{equation*}
\sup_{B_{r}} u\le c_{\varepsilon}\left(\mean{B_{2r}}u^{p_{0}}\,dx\right)^{1/p_{0}} + \varepsilon g_{B_{r}}^{-1}(T(u_{+};r,2r)).
\end{equation*}
We now apply Lemmas~\ref{weak.harnack} and \ref{tail.pm} to the first and the second terms in the right-hand side, respectively, thereby obtaining
\begin{equation*}
\sup_{B_{r}}u \le c\varepsilon\sup_{B_{r}}u + c_{\varepsilon}\inf_{B_{r}}u + c_{\varepsilon}g_{B_{r}}^{-1}(T(u_{-};10r,6r)).
\end{equation*}
In this last display, we choose $\delta = 1/(2c)$ and then reabsorb the first term in the right-hand side; this finally yields \eqref{Harnack}.
\end{proof}

\subsection*{Conflict of interest} The authors declare that they have no conflict of interest.

\subsection*{Data availability} Data sharing not applicable to this article as no datasets were generated or analyzed during the current study.

\providecommand{\bysame}{\leavevmode\hbox to3em{\hrulefill}\thinspace}
\providecommand{\MR}{\relax\ifhmode\unskip\space\fi MR }
\providecommand{\MRhref}[2]{%
  \href{http://www.ams.org/mathscinet-getitem?mr=#1}{#2}
}
\providecommand{\href}[2]{#2}

\end{document}